\documentclass[a4paper]{amsart}

\usepackage{epic, eepic, amsfonts, latexsym, amssymb, graphicx,
multicol, mathrsfs, color, amscd, verbatim, paralist,
xspace, url, euscript, stmaryrd,  amsmath, enumitem,
bbold, multirow, tikz, mathtools}
\usepackage[all,pdf,cmtip]{xy}

\usepackage[colorlinks, linkcolor=blue, citecolor=magenta, urlcolor=cyan]{hyperref}

\usepackage{chngcntr}
\usepackage{apptools}

\def\hat{\widehat}
\renewcommand\bar{\overline}

\def\RR{{\mathbb R}}
\def\CC{{\mathbb C}}

\def\hat{\widehat}

\def\GL{\mathop{\rm GL}\nolimits}

\def\O{\mathop{\rm O}\nolimits}
\def\SO{\mathop{\rm SO}\nolimits}

\def\Aut{\mathop{\rm Aut}\nolimits}
\def\Im{\mathop{\rm Im}\nolimits}
\def\Re{\mathop{\rm Re}\nolimits}
\def\ad{\mathop{\rm ad}\nolimits}

\def\Aff{\mathop{\rm Aff}\nolimits}

\def\tr{\mathop{\rm tr}\nolimits}

\def\qed{{\hfill $\Box$}}
\newtheorem{theorem}{THEOREM}[section]

\newtheorem{proposition}[theorem]{Proposition}

\theoremstyle{definition}

\newtheorem{lemma}[theorem]{Lemma}

\theoremstyle{remark}
\newtheorem{remark}[theorem]{Remark}

\renewcommand{\theequation}{\arabic{section}.\arabic{equation}}

\makeatletter
\def\blfootnote{\xdef\@thefnmark{}\@footnotetext}
\makeatother

\begin{document}

\title[Hyperbolic manifolds with high-dimensional automorphism group]{Homogeneous Kobayashi-hyperbolic manifolds\\ with high-dimensional group of\\ holomorphic automorphisms}
\author[Isaev]{Alexander Isaev}

\address{Mathematical Sciences Institute\\
Australian National University\\
Canberra, Acton, ACT 2601, Australia}
\email{alexander.isaev@anu.edu.au}

\maketitle

\thispagestyle{empty}

\pagestyle{myheadings}

\begin{abstract}We determine all connected homogeneous Kobayashi-hyperbolic manifolds of dimension $n\ge 2$ whose holomorphic automorphism group has dimension $n^2-2$. This result complements an existing classification for automorphism group dimension $n^2-1$ and greater obtained without the homogeneity assumption.
\end{abstract}
\vspace{-0.4cm}

\section{Introduction}\label{intro}
\setcounter{equation}{0}

Let $M$ be a connected complex manifold of dimension $n\ge 2$. It is called \emph{Kobayashi-hyperbolic} if the Kobayashi pseudodistance $K_M$ on $M$ is in fact a distance, i.e., for $p,q\in M$ the identity $K_M(p,q)=0$ implies $p=q$. For example, any bounded domain in complex space $\CC^n$ is Kobayashi-hyperbolic. Such manifolds are of substantial interest in complex analysis and geometry as they possess many attractive properties (see monographs \cite{K1}, \cite{K2} for details). In particular, if $M$ is Kobayashi-hyperbolic, the group $\Aut(M)$ of its holomorphic automorphisms is a (real) Lie group in the compact-open topology (see \cite[Chapter V, Theorem 2.1]{K1}). One way to prove this is by observing that the action of $\Aut(M)$ on $M$ is proper, which implies that $\Aut(M)$ is locally compact hence a Lie transformation group (see \cite{I5} for details). 

Let $d(M):=\dim\Aut(M)$. It is a classical fact that $d(M)\le n^2+2n$ and the equality holds if and only if $M$ is biholomorphic to the unit ball $B^n\subset\CC^n$ (see \cite[Chapter V, Theorem 2.6]{K1}). In papers \cite{I1}, \cite{I2}, \cite{I4}, \cite{IK} we classified all Kobayashi-hyperbolic manifolds with $n^2-1\le d(M)< n^2+2n$ (see also \cite{I3}, \cite{I5} for alternative expositions of these results). In particular, it turned out that no manifolds satisfy $n^2+3\le d(M)<n^2+2n$, i.e., that the possible values of $d(M)$ contain a lacuna located between $n^2+2$ and the maximal dimension $n^2+2n$. We note that the lacunary behavior of $d(M)$ is analogous to that of the dimension of the isometry group of a Riemannian manifold (see, e.g., \cite{I5} for a brief survey). 

Our classification has turned out to be rather useful in applications (see, e.g., \cite{GIL}, \cite{V}), and it would be desirable to extend it beyond the case $d(M)=n^2-1$. However, the value $n^2-2$ is critical in the sense that one cannot hope to obtain a full explicit description of Kobayashi-hyperbolic manifolds for $d(M)=n^2-2$ and all $n\ge 2$. Indeed, a generic Reinhardt domain in $\CC^2$ has a 2-dimensional automorphism group, so no reasonable classification exists for $n=2$ (see \cite[pp.~6--7]{I3} for a precise argument). Furthermore, producing an explicit classification for $n\ge 3$ appears to be out of reach either as the amount of work required to deal even with the easier case $d(M)=n^2-1$ is already rather substantial (see \cite{I1}, \cite{I4}). 

At the same time, some hope remains in the situation when $M$ is \emph{homogeneous}, i.e., when the action of $\Aut(M)$ on $M$ is transitive. Such manifolds are of general interest, and we focus on them in this paper. Specifically, in the following theorem we obtain a description of homogeneous Kobayashi-hyperbolic manifolds satisfying\linebreak $d(M)=n^2-2$. It is remarkably easy to state (although not so easy to prove):

\begin{theorem}\label{main}
Let $M$ be a homogeneous Kobayashi-hyperbolic manifold with $d(M)=n^2-2$. Then $M$ is biholomorphic either to $B^2\times B^1\times B^1$ {\rm (}here $n=4${\rm )} or to $B^3\times B^2$ {\rm (}here $n=5${\rm )}. 
\end{theorem}

Combined with the classical fact for $d(M)=n^2+2n$ mentioned above and \cite[Theorem 2.2]{I3}, this result yields:

\begin{theorem}\label{combined}
Let $M$ be a homogeneous Kobayashi-hyperbolic manifold satisfying $n^2-2\le d(M)\le n^2+2n$. Then $M$ is biholomorphic either to a suitable product of balls or to a suitable symmetric bounded domain of type {\rm (IV)}. Specifically, the following products of balls are possible:
\begin{itemize}

\item[{\rm (i)}] $B^n$ {\rm (}here $d(M)=n^2+2n${\rm )},
\vspace{0.1cm}

\item[{\rm (ii)}] $B^{n-1}\times B^1$ {\rm (}here $d(M)=n^2+2${\rm )},
\vspace{0.1cm}

\item[{\rm (iii)}] $B^1\times B^1\times B^1$ {\rm (}here $n=3$, $d(M)=9=n^2${\rm )},
\vspace{0.1cm}

\item[{\rm (iv)}] $B^2\times B^2$ {\rm (}here $n=4$, $d(M)=16=n^2${\rm )},
\vspace{0.1cm}

\item[{\rm (v)}] $B^2\times B^1\times B^1$ {\rm (}here $n=4$, $d(M)=14=n^2-2${\rm )},
\vspace{0.1cm}

\item[{\rm (vi)}] $B^3\times B^2$ {\rm (}here $n=5$, $d(M)=23=n^2-2${\rm )},  
\end{itemize}

\noindent and the following symmetric bounded domains of type {\rm (IV)} {\rm (}written in tube form{\rm )} are possible:

\begin{itemize}

\item[{\rm (vii)}] the domain of type {\rm (}$\hbox{{\rm IV}}_3${\rm )}
\begin{equation}
T_3:=\left\{(z_1,z_2,z_3)\in\CC^3:(\Im z_1)^2-(\Im z_2)^2-(\Im z_3)^2>0,\,\,\Im z_1>0\right\}\label{domaint3}
\end{equation}
{\rm (}here $n=3$, $d(M)=10=n^2+1${\rm )},
\vspace{0.3cm}

\item[{\rm (viii)}] the domain of type {\rm (}$\hbox{{\rm IV}}_4${\rm )}
\begin{equation}
\begin{array}{l}
T_4:=\left\{(z_1,z_2,z_3,z_4)\in\CC^4:(\Im z_1)^2-(\Im z_2)^2-\right.\\
\vspace{-0.3cm}\\
\hspace{5.5cm}\left.(\Im z_3)^2-(\Im z_4)^2>0,\,\,\Im z_1>0\right\}
\end{array}\label{domaint4}
\end{equation} 
{\rm (}here $n=4$, $d(M)=15=n^2-1${\rm )}.

\end{itemize}
\end{theorem}

The proof of Theorem \ref{main} is given in Section \ref{proof} and is based on reduction to the case of the so-called \emph{Siegel domains of the second kind} introduced by I. Pyatetskii-Shapiro at the end of the 1950s in relation to problems in the theory of automorphic functions (see Section \ref{prelim} for the definition and some properties of such domains). Indeed, in the seminal work \cite{VGP-S} it was shown that every homogeneous bounded domain in $\CC^n$ is biholomorphic to an affinely homogeneous Siegel domain of the second kind. Furthermore, in \cite{N} this result was extended to arbitrary homogeneous Kobayashi-hyperbolic manifolds, which solved a problem posed in \cite{K1} (see Problem 8 on p.~127 therein). The proof of Theorem \ref{main} then follows by analyzing the (graded) Lie algebra of the automorphism group of a Siegel domain of the second kind, which was described in \cite{KMO}, \cite[Chapter V, \S 1--2]{S}.  The analysis of this Lie algebra also yields additional facts (included in the appendix) that can be utilized for extending the classifications in Theorems \ref{main}, \ref{combined} beyond the critical automorphism group dimension $n^2-2$. Our arguments show that the formulas for the graded components of the algebra provided in \cite{KMO}, \cite{S} can be quite useful in applications despite the fact that explicit calculations involving these formulas are rarely seen in the literature.

{\bf Acknowledgement.} This work is supported by the Australian Research\linebreak Council.

\section{Preliminaries on Convex Cones and Siegel Domains\\ of the Second Kind}\label{prelim}
\setcounter{equation}{0}

In this section we define Siegel domains of the second kind and collect their properties as required for our proof of Theorem \ref{main} in the next section.

First of all, an open subset $\Omega\subset\RR^k$ is called an \emph{open convex cone} if it is closed with respect to taking linear combinations of its elements with positive coefficients. Such a cone $\Omega$ is called \emph{{\rm (}linearly{\rm )} homogeneous} if the group
$$
G(\Omega):=\{A\in\GL_k(\RR):A\Omega=\Omega\}
$$
of linear automorphisms of $\Omega$ acts transitively on it. Clearly, $G(\Omega)$ is a closed subgroup of $\GL_k(\RR)$, and we denote by ${\mathfrak g}(\Omega)\subset{\mathfrak{gl}}_k(\RR)$ its Lie algebra.

We will be interested in open convex cones not containing entire lines. For such cones the dimension of ${\mathfrak g}(\Omega)$ admits a useful estimate.

\begin{lemma}\label{ourlemma}\it Let $\Omega\subset\RR^k$ be an open convex cone not containing a line. Then
\begin{equation}
\dim {\frak g}(\Omega)\le\displaystyle\frac{k^2}{2}-\frac{k}{2}+1.\label{conineq}
\end{equation}
\end{lemma}

\begin{proof} Fix a point ${\mathbf x}\in \Omega$ and consider its isotropy subgroup $G_{{\mathbf x}}(\Omega)\subset G(\Omega)$. This subgroup is compact since it leaves invariant the bounded open set $\Omega\cap({\mathbf x}-\Omega)$. Therefore, changing variables in $\RR^k$ if necessary, we can assume that $G_{{\mathbf x}}(\Omega)$ lies in the orthogonal group $\O_k(\RR)$. The group $\O_k(\RR)$ acts transitively on the sphere of radius $||{\mathbf x}||$ in $\RR^k$, and the isotropy subgroup $I_{{\mathbf x}}$ of ${\mathbf x}$ under the $\O_k(\RR)$-action is isomorphic to $\O_{k-1}(\RR)$. Since $G_{{\mathbf x}}(\Omega)\subset I_{{\mathbf x}}$, we have
$$
\dim G_{{\mathbf x}}(\Omega)\le\dim I_{{\mathbf x}}=\displaystyle\frac{k^2}{2}-\frac{3k}{2}+1,
$$
which implies inequality (\ref{conineq}).\end{proof}

Next, let
$$
H:\CC^m\times\CC^m\to\CC^k
$$
be a Hermitian form on $\CC^m$ with values in $\CC^k$, where we assume that $H(w,w')$ is linear in $w'$ and anti-linear in $w$. For an open convex cone $\Omega\subset\RR^k$, the form $H$ is called \emph{$\Omega$-Hermitian} if $H(w,w)\in\overline{\Omega}\setminus\{0\}$ for all non-zero $w\in\CC^m$. Observe that if $\Omega$ contains no lines and $H$ is $\Omega$-Hermitian, then there exists a positive-definite linear combination of the components of $H$.

Now, a Siegel domain of the second kind in $\CC^n$ is an unbounded domain of the form
$$
S(\Omega,H):=\left\{(z,w)\in\CC^k\times\CC^{n-k}:\Im z-H(w,w)\in \Omega\right\}
$$
for some $1\le k\le n$, some open convex cone $\Omega\subset\RR^k$ not containing a line, and some $\Omega$-Hermitian form $H$ on $\CC^{n-k}$. For $k=n$ we have $H=0$, so in this case $S(\Omega,H)$ is the tube domain       
$$
\left\{z\in\CC^n:\Im z\in \Omega\right\}.
$$
Such tube domains are often called \emph{Siegel domains of the first kind}. At the other extreme, when $k=1$, the domain $S(\Omega,H)$ is linearly equivalent to
$$
\left\{(z,w)\in\CC\times\CC^{n-1}:\Im z-||w||^2>0\right\},
$$ 
which is an unbounded realization of the unit ball $B^n$ (see \cite[p.~31]{R}). In fact, any Siegel domain of the second kind is biholomorphic to a bounded domain (see \cite[pp.~23--24]{P-S}), hence is Kobayashi-hyperbolic.

Next, the holomorphic affine automorphisms of Siegel domains of the second kind are described as follows (see \cite[pp.~25-26]{P-S}):

\begin{theorem}\label{Siegelaffautom}
Any holomorphic affine automorphism of $S(\Omega,H)$ has the form
$$
\begin{array}{lll}
z&\mapsto & Az+a+2iH(b,Bw)+iH(b,b),\\
\vspace{-0.3cm}\\
w&\mapsto & Bw+b,
\end{array}
$$
with $a\in\RR^k$, $b\in\CC^{n-k}$, $A\in G(\Omega)$, $B\in\GL_{n-k}(\CC)$, where
\begin{equation}
AH(w,w')=H(Bw,Bw')\label{assoc}
\end{equation}
for all $w,w'\in\CC^{n-k}$.
\end{theorem}

A domain $S(\Omega,H)$ is called \emph{affinely homogeneous} if the group $\Aff(S(\Omega,H))$ of its holomorphic affine automorphisms acts on $S(\Omega,H)$ transitively. Denote by $G(\Omega,H)$ the subgroup of $G(\Omega)$ that consists of all transformations $A\in G(\Omega)$ as in Theorem \ref{Siegelaffautom}, namely, of all elements $A\in G(\Omega)$ for which there exists $B\in\GL_{n-k}(\CC)$ such that (\ref{assoc}) holds. By \cite[Lemma 1.1]{D}, the subgroup $G(\Omega,H)$ is closed in $G(\Omega)$. It is easy to deduce from Theorem \ref{Siegelaffautom} that if $S(\Omega,H)$ is affinely homogeneous, the action of $G(\Omega,H)$ (hence that of its identity component $G(\Omega,H)^{\circ}$) is transitive on $\Omega$ (see, e.g., \cite[proof of Theorem 8]{KMO}), so the cone $\Omega$ is homogeneous. Conversely, if $G(\Omega,H)$ acts on $\Omega$ transitively, the domain $S(\Omega,H)$ is affinely homogeneous.

As shown in \cite{VGP-S}, \cite{N}, every homogeneous Kobayashi-hyperbolic manifold is biholomorphic to an affinely homogeneous Siegel domain of the second kind. Such a realization is unique up to affine transformations; in general, if two Siegel domains of the second kind are biholomorphic to each other, they are also equivalent by means of a linear transformation of special form (see \cite[Theorem 11]{KMO}). The result of \cite{VGP-S}, \cite{N} is the basis of our proof of Theorem \ref{main} in the next section.

In addition, our proof relies on a description of the Lie algebra of the group $\Aut(S(\Omega,H))$ of an arbitrary Siegel domain of the second kind $S(\Omega,H)$. This algebra is isomorphic to the (real) Lie algebra of complete holomorphic vector fields on $S(\Omega,H)$, which we denote by ${\mathfrak g}(S(\Omega,H))$ or, when there is no fear of confusion, simply by ${\mathfrak g}$. This algebra has been extensively studied. In particular, we have (see \cite[Theorems 4 and 5]{KMO}):

\begin{theorem}\label{kmoalgebradescr}
The algebra ${\mathfrak g}={\mathfrak g}(S(\Omega,H))$ admits a grading
$$
{\mathfrak g}={\mathfrak g}_{-1}\oplus{\mathfrak g}_{-1/2}\oplus{\mathfrak g}_0\oplus{\mathfrak g}_{1/2}\oplus{\mathfrak g}_1,
$$
with ${\mathfrak g}_{\nu}$ being the eigenspace with eigenvalue $\nu$ of $\ad\partial$, where
$\displaystyle\partial:=z\cdot\frac{\partial}{\partial z}+\frac{1}{2}w\cdot\frac{\partial}{\partial w}$. 
Here
$$
\begin{array}{ll}
{\mathfrak g}_{-1}=\displaystyle\left\{a\cdot\frac{\partial}{\partial z}:a\in\RR^k\right\},&\dim {\mathfrak g}_{-1}=k,\\
\vspace{-0.1cm}\\
{\mathfrak g}_{-1/2}=\displaystyle\left\{2i H(b,w)\cdot\frac{\partial}{\partial z}+b\cdot\frac{\partial}{\partial w}:b\in\CC^{n-k}\right\},&\dim {\mathfrak g}_{-1/2}=2(n-k),
\end{array}
$$
and ${\mathfrak g}_0$ consists of all vector fields of the form
\begin{equation}
(Az)\cdot\frac{\partial}{\partial z}+(Bw)\cdot\frac{\partial}{\partial w},\label{g0}
\end{equation}
with $A\in{\mathfrak g}(\Omega)$, $B\in{\mathfrak{gl}}_{n-k}(\CC)$ and
\begin{equation}
AH(w,w')=H(Bw,w')+H(w,Bw')\label{assoc1}
\end{equation}
for all $w,w'\in\CC^{n-k}$. Furthermore, one has
\begin{equation} 
\dim {\mathfrak g}_{1/2}\le 2(n-k),\qquad \dim {\mathfrak g}_1\le k.\label{estimm}
\end{equation}
\end{theorem}

It is then clear that the matrices $A$ that appear in (\ref{g0}) form the Lie algebra of $G(\Omega,H)$ and that ${\mathfrak g}_{-1}\oplus{\mathfrak g}_{-1/2}\oplus{\mathfrak g}_0$ is isomorphic to the Lie algebra of the group $\Aff(S(\Omega,H))$ (compare conditions (\ref{assoc}) and (\ref{assoc1})). 

Following \cite{S}, for a pair of matrices $A,B$ satisfying (\ref{assoc1}) we say that $B$ is \emph{associated to $A$} (with respect to $H$). Let ${\mathcal L}$ be the (real) subspace of ${\mathfrak{gl}}_{n-k}(\CC)$ of all matrices associated to the zero matrix in ${\mathfrak g}(\Omega)$, i.e., matrices skew-Hermitian with respect to each component of $H$. Set $s:=\dim {\mathcal L}$. Then we have
\begin{equation}
\dim {\mathfrak g}_0\le s+\dim {\mathfrak g}(\Omega).\label{estim1}
\end{equation}
By Theorem \ref{kmoalgebradescr} and inequality (\ref{estim1}) one obtains
\begin{equation}
d(S(\Omega,H))\le k+2(n-k)+s+\dim {\mathfrak g}(\Omega)+\dim {\mathfrak g}_{1/2}+ \dim {\mathfrak g}_1,\label{estim 8}
\end{equation}
which, combined with (\ref{estimm}) leads to
\begin{equation}
d(S(\Omega,H))\le 2k+4(n-k)+s+\dim {\mathfrak g}(\Omega).\label{estim2}
\end{equation}
Further, since there exists a positive-definite linear combination ${\mathbf H}$ of the components of the Hermitian form $H$, the subspace ${\mathcal L}$ lies in the Lie algebra of matrices skew-Hermitian with respect to ${\mathbf H}$, thus
\begin{equation}
s\le (n-k)^2.\label{ests}
\end{equation}
By (\ref{ests}), inequality (\ref{estim2}) yields
\begin{equation}
d(S(\Omega,H))\le 2k+4(n-k)+(n-k)^2+\dim {\mathfrak g}(\Omega).\label{estim3}
\end{equation}
Combining (\ref{estim3}) with (\ref{conineq}), we deduce the following useful upper bound:
\begin{equation}
d(S(\Omega,H))\le\displaystyle\frac{3k^2}{2}-k\left(2n+\frac{5}{2}\right)+n^2+4n+1.\label{estim4}
\end{equation}

Next, by \cite[Chapter V, Proposition 2.1]{S} the component ${\mathfrak g}_{1/2}$ of the Lie algebra ${\mathfrak g}={\mathfrak g}(S(\Omega,H))$ is described as follows:

\begin{theorem}\label{descrg1/2}
The subspace ${\mathfrak g}_{1/2}$ consists of all vector fields of the form
$$
2iH(\Phi(\bar z),w)\cdot\frac{\partial}{\partial z}+(\Phi(z)+c(w,w))\cdot\frac{\partial}{\partial w},
$$
where $\Phi:\CC^k\to\CC^{n-k}$ is a $\CC$-linear map such that for every ${\mathbf w}\in\CC^{n-k}$ one has
\begin{equation}
\Phi_{{\mathbf w}}:=\left[x\mapsto\Im H({\mathbf w},\Phi(x)),\,\, x\in\RR^k\right]\in{\mathfrak g}(\Omega),\label{Phiw0}
\end{equation}
and $c:\CC^{n-k}\times\CC^{n-k}\to\CC^{n-k}$ is a symmetric $\CC$-bilinear form on $\CC^{n-k}$ with values in $\CC^{n-k}$ satisfying the condition
\begin{equation}
H(w,c(w',w'))=2iH(\Phi(H(w',w)),w')\label{cond1}
\end{equation}
for all $w,w'\in\CC^{n-k}$. 
\end{theorem}

Further, by \cite[Chapter V, Proposition 2.2]{S}, the component ${\mathfrak g}_1$ of ${\mathfrak g}={\mathfrak g}(S(\Omega,H))$ admits the following description:

\begin{theorem}\label{descrg1}
The subspace ${\mathfrak g}_1$ consists of all vector fields of the form
$$
a(z,z)\cdot\frac{\partial}{\partial z}+b(z,w)\cdot\frac{\partial}{\partial w},
$$
where $a:\RR^k\times\RR^k\to\RR^k$ is a symmetric $\RR$-bilinear form on $\RR^k$ with values in $\RR^k$ {\rm (}which we extend to a symmetric $\CC$-bilinear form on $\CC^k$ with values in $\CC^k${\rm )} such that for every ${\mathbf x}\in\RR^k$ one has
\begin{equation}
A_{{\mathbf x}}:=\left[x\mapsto a({\mathbf x},x),\,\,x\in\RR^k\right]\in{\mathfrak g}(\Omega),\label{idents1}
\end{equation}
and $b:\CC^k\times\CC^{n-k}\to\CC^{n-k}$ is a $\CC$-bilinear map such that, if for ${\mathbf x}\in\RR^k$ one sets
\begin{equation}
B_{{\mathbf x}}:=\left[w\mapsto\frac{1}{2}b({\mathbf x},w),\,\,w\in\CC^{n-k}\right],\label{idents2}
\end{equation}
the following conditions are satisfied:
\begin{itemize}

\item[{\rm (i)}] $B_{{\mathbf x}}$ is associated to $A_{{\mathbf x}}$ and $\Im\tr B_{{\mathbf x}}=0$ for all ${\mathbf x}\in\RR^k$,
\vspace{0.1cm}

\item[{\rm (ii)}] for every pair ${\mathbf w},{\mathbf w}'\in\CC^{n-k}$ one has
$$
B_{{\mathbf w},{\mathbf w}'}:=\left[x\mapsto\Im H({\mathbf w'},b(x,{\mathbf w})),\,\,x\in\RR^k\right]\in{\mathfrak g}(\Omega),
$$

\item[{\rm (iii)}] $H(w,b(H(w',w''),w''))=H(b(H(w'',w),w'),w'')$ for all $w,w',w''\in\CC^{n-k}$.

\end{itemize}
\end{theorem}

Next, let us recall the well-known classification, up to linear equivalence, of homogeneous convex cones not containing lines in dimensions $k=2,3,4$ (see, e.g., \cite[pp.~38--41]{KT}), which will be also required for our proof of Theorem \ref{main}:

\begin{itemize}

\item [$k=2$:] 

\begin{itemize}

\item[]

$\Omega_1:=\left\{(x_1,x_2)\in\RR^2:x_1>0,\,\,x_2>0\right\}$, where the algebra ${\mathfrak g}(\Omega_1)$ consists of all diagonal matrices, hence $\dim {\mathfrak g}(\Omega_1)=2$,
\end{itemize}
\vspace{0.3cm}

\item [$k=3$:] 

\begin{itemize}

\item[(i)] $\Omega_2:=\left\{(x_1,x_2,x_3)\in\RR^3:x_1>0,\,\,x_2>0,\,\,x_3>0\right\}$, where the algebra ${\mathfrak g}(\Omega_2)$ consists of all diagonal matrices, hence $\dim {\mathfrak g}(\Omega_2)=3$,
\vspace{0.1cm}

\item[(ii)] $\Omega_3:=\left\{(x_1,x_2,x_3)\in\RR^3:x_1^2-x_2^2-x_3^2>0,\,\,x_1>0\right\}$, where one has ${\mathfrak g}(\Omega_3)={\mathfrak c}({\mathfrak{gl}}_3(\RR))\oplus{\mathfrak o}_{1,2}$, hence $\dim {\mathfrak g}(\Omega_3)=4$; here for any Lie algebra ${\mathfrak h}$ we denote by ${\mathfrak c}({\mathfrak h})$ its center,

\end{itemize}
\vspace{0.3cm}

\item [$k=4$:] 

\begin{itemize}

\item[(i)] $\Omega_4:=\left\{(x_1,x_2,x_3,x_4)\in\RR^4:x_1>0,\,\,x_2>0,\,\,x_3>0,\,\,x_4>0\right\}$, where the algebra ${\mathfrak g}(\Omega_4)$ consists of all diagonal matrices, hence we have $\dim {\mathfrak g}(\Omega_4)=4$,
\vspace{0.1cm}

\item[(ii)] $\Omega_5:=\left\{(x_1,x_2,x_3,x_4)\in\RR^4: x_1^2-x_2^2-x_3^2>0,\,\,x_1>0,\,\,x_4>0\right\}$, where the algebra ${\mathfrak g}(\Omega_5)=\left({\mathfrak c}({\mathfrak{gl}}_3(\RR))\oplus{\mathfrak o}_{1,2}\right)\oplus\RR$ consists of block-diagonal matrices with blocks of sizes $3\times 3$ and $1\times 1$ corresponding to the two summands, hence $\dim {\mathfrak g}(\Omega_5)=5$,
\vspace{0.1cm}

\item[(iii)] $\Omega_6:=\left\{(x_1,x_2,x_3,x_4)\in\RR^4: x_1^2-x_2^2-x_3^2-x_4^2>0,\,\,x_1>0\right\}$, where ${\mathfrak g}(\Omega_6)={\mathfrak c}({\mathfrak{gl}}_4(\RR))\oplus{\mathfrak o}_{1,3}$, hence $\dim {\mathfrak g}(\Omega_6)=7$.
\end{itemize}
\end{itemize}

In \cite{C}, \'E. Cartan found all homogeneous bounded domains in $\CC^2$ and $\CC^3$. We conclude this section with a short proof of Cartan's theorem (extended to the case of Kobayashi-hyperbolic manifolds) based on Siegel domains of the second kind.

\begin{theorem}\label{cartansclass}\hfill 
\begin{itemize}

\item[{\rm (1)}] Every homogeneous Kobayashi-hyperbolic manifold of dimension $2$ is biholomorphic to one of

\begin{itemize}

\item[{\rm(i)}] $B^2$,
\vspace{0.1cm}

\item[{\rm(ii)}] $B^1\times B^1$.
\end{itemize}
\vspace{0.3cm}
 
\item[{\rm (2)}] Every homogeneous Kobayashi-hyperbolic manifold of dimension $3$ is biholomorphic to one of

\begin{itemize}

\item[{\rm(i)}] $B^3$,
\vspace{0.1cm}

\item[{\rm(ii)}] $B^2\times B^1$,
\vspace{0.1cm}

\item[{\rm(iii)}] $B^1\times B^1\times B^1$,
\vspace{0.1cm}

\item[{\rm(iv)}] the tube domain $T_3$ defined in {\rm (\ref{domaint3})}.

\end{itemize}
\end{itemize}

\end{theorem}

\begin{proof}
Let $M$ be a homogeneous Kobayashi-hyperbolic manifold of dimension $n$. By \cite{VGP-S}, \cite{N}, the manifold $M$ is biholomorphic to a Siegel domain of the second kind $S(\Omega,H)$. If $k=1$, then $S(\Omega,H)$ is biholomorphic to $B^n$, so we assume that $k\ge 2$. 

If $n=2$, then $k=2$, hence after a linear change of variables $S(\Omega,H)$ becomes
$$
\left\{z\in\CC^2: \Im z\in\Omega_1\right\}
$$
and therefore is biholomorphic to $B^1\times B^1$. This establishes Part (1).

Assume that $n=3$ and suppose first that $k=3$. Then after a linear change of variables $S(\Omega,H)$ turns into one of the domains
$$
\begin{array}{l}
\left\{z\in\CC^3: \Im z\in\Omega_2\right\},\\
\vspace{-0.1cm}\\
\left\{z\in\CC^3: \Im z\in\Omega_3\right\}
\end{array}
$$
and therefore is biholomorphic to either $B^1\times B^1\times B^1$ or the tube domain $T_3$. 

Let now $k=2$. In this case, after a linear change of variables $S(\Omega,H)$ becomes 
$$
D:=\left\{(z,w)\in\CC^2\times\CC: \Im z-v|w|^2\in\Omega_1\right\},
$$
where $v=(v_1,v_2)$ is a non-zero vector in $\RR^2$ with non-negative components. Let us compute the group $G(\Omega_1,v|w|^2)$. It consists of all non-degenerate diagonal matrices
$$
\left(\begin{array}{cc}
\lambda_1 & 0\\
0 &  \lambda_2
\end{array}
\right),\,\, 
\left(\begin{array}{cc}
0 & \mu_1\\
\mu_2 & 0
\end{array}
\right)
$$ 
such that $\lambda_1>0$, $\lambda_2>0$, $\lambda_1v_1=\rho v_1$, $\lambda_2v_2=\rho v_2$ and $\mu_1>0$, $\mu_2>0$, $\mu_1v_2=\eta v_1$, $\mu_2v_1=\eta v_2$ for some $\rho,\eta>0$. Hence if $v_1\ne 0$, $v_2\ne 0$ we have
\begin{equation}
G(\Omega_1,v|w|^2)=\left\{\left(\begin{array}{cc}
\rho & 0\\
0 & \rho
\end{array}
\right),\,\,
\left(\begin{array}{cc}
0 & \displaystyle\eta\frac{v_1}{v_2}\\
\displaystyle\eta\frac{v_2}{v_1} & 0
\end{array}
\right)\,\,\hbox{with}\,\,\rho,\,\eta>0\right\},\label{gomega}
\end{equation}
and it follows that the action of $G(\Omega_1,v|w|^2)$ is not transitive on $\Omega_1$. This contradiction implies that exactly one of $v_1,v_2$ is non-zero, hence $D$ is linearly equivalent to the domain
$$
\left\{(z,w)\in\CC^2\times\CC:\Im z_1-|w|^2>0,\,\, \Im z_2>0\right\},
$$
which is biholomorphic to $B^2\times B^1$. This proves Part (2). \end{proof}

\section{Proof of Theorem \ref{main}}\label{proof}
\setcounter{equation}{0}

By \cite{VGP-S}, \cite{N}, the manifold $M$ is biholomorphic to a Siegel domain of the second kind $S(\Omega,H)$. Since for each domain listed in Theorem \ref{cartansclass} the dimension of its automorphism group is greater than $n^2-2$, it follows that $n\ge 4$. Also, as $M$ is not biholomorphic to $B^n$, we have $k\ge 2$.

Next, the following lemma rules out a large number of the remaining possibilities.

\begin{lemma}\label{n5k3} \it For $n\ge 5$ one cannot have $k\ge 3$.
\end{lemma}

\begin{proof} We will show that for $n\ge 5$, $k\ge 3$ the right-hand side of inequality (\ref{estim4}) is strictly less than $n^2-2$, i.e., that for such $n,k$ the following holds:
$$
\frac{3k^2}{2}-\left(2n+\frac{5}{2}\right)k+4n+3<0.
$$
In order to see this, let us study the quadratic function
$$
\varphi(t):=\frac{3t^2}{2}-\left(2n+\frac{5}{2}\right)t+4n+3
$$
on the segment $[3,n]$. Its discriminant is
$$
{\mathcal D}:=4n^2-14n-\frac{47}{4},
$$
which is easily seen to be positive for $n\ge 5$. Then the zeroes of $\varphi$ are
$$
\begin{array}{l}
\displaystyle t_1:=\frac{2n+\frac{5}{2}-\sqrt{{\mathcal D}}}{3},\\
\vspace{-0.1cm}\\
\displaystyle t_2:=\frac{2n+\frac{5}{2}+\sqrt{{\mathcal D}}}{3}.
\end{array}
$$

To prove the lemma, it suffices to show that $t_1<3$ and $t_2>n$ for $n\ge 5$. Indeed, the former inequality means that
$$
2n-\frac{13}{2}< \sqrt{{\mathcal D}},
$$
or, equivalently, that
$$
n>\frac{9}{2},
$$
which clearly holds if $n\ge 5$. Further, the inequality $t_2>n$ means that  
$$
n-\frac{5}{2}< \sqrt{{\mathcal D}},
$$
or, equivalently, that  
$$
n^2-3n-6>0,
$$
which is straightforward to verify for $n\ge 5$ as well. \end{proof}

By Lemma \ref{n5k3}, in order to prove the theorem, we in fact need to consider the following three cases: (1) $k=2$, $n\ge 4$, (2) $k=3$, $n=4$, (3) $k=4$, $n=4$.
\vspace{-0.1cm}\\

{\bf Case (1).} Suppose that $k=2$, $n\ge 4$. Here $H=(H_1,H_2)$ is a pair of Hermitian forms on $\CC^{n-2}$. After a linear change of $z$-variables, we may assume that $H_1$ is positive-definite. In this situation, by applying a linear change of $w$-variables, we can simultaneously diagonalize $H_1$, $H_2$ as
$$
H_1(w,w)=||w||^2,\,\,\, H_2(w,w)=\sum_{j=1}^{n-2}\lambda_j|w_j|^2.
$$
If all the eigenvalues of $H_2$ are equal, $S(\Omega,H)$ is linearly equivalent either to
$$ 
D_1:=\left\{(z,w)\in\CC^2\times\CC^{n-2}:\Im z_1-||w||^2>0,\,\,\Im z_2>0\right\},
$$
or to
$$ 
D_2:=\left\{(z,w)\in\CC^2\times\CC^{n-2}:\Im z_1-||w||^2>0,\,\,\Im z_2-||w||^2>0\right\}.
$$ 
The domain $D_1$ is biholomorphic to $B^{n-1}\times B^1$, hence $d(D_1)=n^2+2$, which shows that $S(\Omega,H)$ cannot be equivalent to $D_1$. To deal with $D_2$, let us compute the group $G(\Omega_1,(||w||^2,||w||^2))$. It is straightforward to see that
$$
G(\Omega_1,(||w||^2,||w||^2))=\left\{\left(\begin{array}{cc}
\rho & 0\\
0 & \rho
\end{array}
\right),\,\,
\left(\begin{array}{cc}
0 & \displaystyle\eta\\
\displaystyle\eta & 0
\end{array}
\right)\,\,\hbox{with}\,\,\rho,\,\eta>0\right\},
$$ 
(cf.~(\ref{gomega})), and it follows that the action of $G(\Omega_1,(||w||^2,||w||^2))$ is not transitive on $\Omega_1$. This proves that $S(\Omega,H)$ cannot be equivalent to $D_2$ either. Therefore, $H_2$ has at least one  pair of distinct eigenvalues.

Next, as $\dim{\mathfrak g}(\Omega)=2$, inequality (\ref{estim2}) yields
\begin{equation}
s\ge n^2-4n.\label{estim5}
\end{equation} 
On the other hand, by (\ref{ests}), we have
$$
s\le n^2-4n+4.
$$
More precisely, $s$ is calculated as
\begin{equation}
s=n^2-4n+4-2m,\label{estim7}
\end{equation}
where $m\ge 1$ is the number of pairs of distinct eigenvalues of $H_2$. Indeed, if
$$
B=\left(B_{ij}\right),\,\,\, B_{ij}=-\overline{B_{ji}},\,\,\, i,j=1,\dots,n-2,
$$
is skew-symmetric with respect to $H_1$, the condition of skew-symmetricity with respect to $H_2$ is written as
$$
B_{ij}\lambda_i=-\overline{B}_{ji}\lambda_j,\,\,\, i,j=1,\dots,n-2,
$$
which leads to $B_{ij}=0$ if $\lambda_i\ne\lambda_j$. 

By (\ref{estim5}), (\ref{estim7}) it follows that $1\le m\le 2$, thus we have either $n=4$ and $\lambda_1\ne\lambda_2$ (here $m=1$, $s=2$), or $n=5$ and, upon permutation of $w$-variables, $\lambda_1\ne\lambda_2=\lambda_3$ (here $m=2$, $s=5$). We will now consider these two situations separately.
\vspace{0.1cm}

{\bf Case (1a).} Suppose that $n=4$, $\lambda_1\ne\lambda_2$. Here, after a linear change of variables the domain $S(\Omega,H)$ takes the form
\begin{equation}
\begin{array}{ll}
D_3:=\left\{(z,w)\in\CC^2\times\CC^2:\Im z_1-(\alpha|w_1|^2+\beta|w_2|^2)>0,\right.\\
\vspace{-0.3cm}\\
\hspace{5cm}\left.\Im z_2-(\gamma|w_1|^2+\delta|w_2|^2)>0\right\},
\end{array}\label{domaind3}
\end{equation}
where $\alpha,\beta,\gamma,\delta\ge 0$ and
$$
\det\left(\begin{array}{ll}
\alpha & \beta\\
\gamma & \delta
\end{array}
\right)\ne 0.
$$
We may also assume that $\alpha>0$. If $\beta=\gamma=0$, the domain $D_3$ is biholomorphic to $B^2\times B^2$. Since $d(B^2\times B^2)=16=n^2$, we in fact have $\beta+\gamma>0$.

\begin{lemma}\label{g12d3}\it If $\beta+\gamma>0$, for ${\mathfrak g}={\mathfrak g}(D_3)$ one has ${\mathfrak g}_{1/2}=0$.
\end{lemma}

\begin{proof} We will apply Theorem \ref{descrg1/2} to the cone $\Omega_1$ and the $\Omega_1$-Hermitian form
\begin{equation}
{\mathcal H}(w,w'):=(\alpha\bar w_1w_1'+\beta\bar w_2w_2', \gamma\bar w_1w_1'+\delta\bar w_2w_2').\label{hermform1}
\end{equation} 
Let $\Phi:\CC^2\to\CC^2$ be a $\CC$-linear map:
$$
\Phi(z_1,z_2)=(\varphi_1^1z_1+\varphi^1_2z_2,\varphi^2_1z_1+\varphi^2_2z_2),
$$
where $\varphi^i_j\in\CC$. Fixing ${\mathbf w}\in\CC^2$, for $x\in\RR^2$ we compute
$$
\begin{array}{l}
{\mathcal H}({\mathbf w},\Phi(x))=\left(\alpha\bar {\mathbf w}_1(\varphi^1_1x_1+\varphi^1_2x_2)+\beta\bar {\mathbf w}_2(\varphi^2_1x_1+\varphi^2_2x_2),\right.\\
\vspace{-0.3cm}\\
\hspace{5cm}\left.\gamma\bar {\mathbf w}_1(\varphi^1_1x_1+\varphi^1_2x_2)+\delta \bar {\mathbf w}_2(\varphi^2_1x_1+\varphi^2_2x_2)\right)=\\
\vspace{-0.1cm}\\
\left((\alpha\bar {\mathbf w}_1\varphi^1_1+\beta\bar {\mathbf w}_2\varphi^2_1)x_1+(\alpha\bar {\mathbf w}_1\varphi^1_2+\beta\bar {\mathbf w}_2\varphi^2_2)x_2,\right.\\
\vspace{-0.1cm}\\
\hspace{4.8cm}\left.(\gamma\bar {\mathbf w}_1\varphi^1_1+\delta \bar {\mathbf w}_2\varphi^2_1)x_1+(\gamma\bar {\mathbf w}_1\varphi^1_2+\delta \bar {\mathbf w}_2\varphi^2_2)x_2\right).
\end{array}
$$
Then from formula (\ref{Phiw0}) we see
$$
\begin{array}{l}
\Phi_{{\mathbf w}}(x)=\left((\alpha\Im(\bar {\mathbf w}_1\varphi^1_1)+\beta\Im(\bar {\mathbf w}_2\varphi^2_1))x_1+(\alpha\Im(\bar {\mathbf w}_1\varphi^1_2)+\beta\Im(\bar {\mathbf w}_2\varphi^2_2))x_2,\right.\\
\vspace{-0.3cm}\\
\hspace{2cm}\left.(\gamma\Im(\bar {\mathbf w}_1\varphi^1_1)+\delta\Im(\bar {\mathbf w}_2\varphi^2_1))x_1+(\gamma\Im(\bar {\mathbf w}_1\varphi^1_2)+\delta\Im(\bar {\mathbf w}_2\varphi^2_2))x_2\right).
\end{array}
$$
The condition that this map lies in ${\mathfrak g}(\Omega_1)$ for every ${\mathbf w}\in\CC^2$ means
$$
\begin{array}{l}
\alpha\Im(\bar {\mathbf w}_1\varphi^1_2)+\beta\Im(\bar {\mathbf w}_2\varphi^2_2)\equiv 0,\\
\vspace{-0.3cm}\\
\gamma\Im(\bar {\mathbf w}_1\varphi^1_1)+\delta\Im(\bar {\mathbf w}_2\varphi^2_1)\equiv 0,
\end{array}
$$
which leads to the relations
\begin{equation}
\varphi^1_2=0,\,\,\beta\varphi^2_2=0,\,\,\gamma \varphi^1_1=0,\,\,\delta \varphi^2_1=0\label{idss1} 
\end{equation}
(recall that $\alpha>0$). If each of $\beta$, $\gamma$, $\delta$ is non-zero, we see from (\ref{idss1}) that $\Phi=0$, which by formula (\ref{cond1}) implies ${\mathfrak g}_{1/2}=0$ as required. 

Suppose now that $\beta=0$, hence each of $\gamma$, $\delta$ is non-zero. Then (\ref{idss1}) yields
$$
\varphi^1_1=0,\,\,\varphi^1_2=0,\,\,\varphi^2_1=0.
$$
Thus, $\Phi(z_1,z_2)=(0,\varphi^2_2z_2)$, and for $w,w'\in\CC^2$ we compute
\begin{equation}
2i{\mathcal H}(\Phi({\mathcal H}(w',w)),w')=\left(0,2i\gamma\delta\bar\varphi^2_2\bar w_1w_1'w_2'+2i\delta^2\bar\varphi^2_2\bar w_2(w_2')^2\right).\label{idss2}
\end{equation}

Further, let $c$ be a symmetric $\CC$-bilinear form on $\CC^2$ with values in $\CC^2$:
$$
c(w,w)=\left(c^1_{11}w_1^2+2c^1_{12}w_1w_2+c^1_{22}w_2^2,c^2_{11}w_1^2+2c^2_{12}w_1w_2+c^2_{22}w_2^2\right),
$$
where $c^k_{ij}\in\CC$. Then for $w,w'\in\CC^2$ we have
\begin{equation}
\begin{array}{l}
{\mathcal H}(w,c(w',w'))=\left(\alpha\bar w_1(c^1_{11}(w_1')^2+2c^1_{12}w_1'w_2'+c^1_{22}(w_2')^2),\right.\\
\vspace{-0.3cm}\\
\hspace{1cm}\left.\gamma\bar w_1(c^1_{11}(w_1')^2+2c^1_{12}w_1'w_2'+c^1_{22}(w_2')^2)+\right.\\
\vspace{-0.3cm}\\
\hspace{3cm}\left.\delta\bar w_2(c^2_{11}(w_1')^2+2c^2_{12}w_1'w_2'+c^2_{22}(w_2')^2)\right).
\end{array}\label{idss3}
\end{equation} 
Comparing the right-hand sides of (\ref{idss2}) and (\ref{idss3}) for arbitrary $w,w'$ as required by condition (\ref{cond1}), we see that $c^1_{ij}=0$ for all $i,j$ hence $\varphi^2_2=0$ and therefore $\Phi=0$, $c=0$. Thus for $\beta=0$ we again have ${\mathfrak g}_{1/2}=0$ as claimed. 

The cases $\gamma=0$ and $\delta=0$ are obtained from the case $\beta=0$ by permutation of variables.\end{proof}

By estimate (\ref{estim 8}), the second inequality in (\ref{estimm}), and Lemma \ref{g12d3}, we see
\begin{equation}
d(D_3)\le 12<14=n^2-2\label{essstimm1} 
\end{equation}
(recall that $s=2$). This shows that $S(\Omega,H)$ cannot in fact be equivalent to $D_3$, so Case (1a) contributes nothing to the classification of homogeneous Kobayashi-hyperbolic $n$-dimensional manifolds with automorphism group dimension $n^2-2$.

\begin{remark}\label{remg1d3}
In Proposition \ref{g1d3} in the appendix we prove that for $\beta+\gamma>0$ the component ${\mathfrak g}_1$ of the algebra ${\mathfrak g}={\mathfrak g}(D_3)$ is also zero, which improves estimate (\ref{essstimm1}) to $d(D_3)\le 10$.
\end{remark}
\vspace{0.1cm}

{\bf Case (1b).} Suppose that $n=5$ and $\lambda_1\ne\lambda_2=\lambda_3$. Here, after a linear change of variables the domain $S(\Omega,H)$ takes the form
\begin{equation}
\begin{array}{ll}
D_4:=\left\{(z,w)\in\CC^2\times\CC^3:\Im z_1-(\alpha|w_1|^2+\beta|w_2|^2+\beta|w_3|^2)>0,\right.\\
\vspace{-0.3cm}\\
\hspace{4cm}\left.\Im z_2-(\gamma|w_1|^2+\delta|w_2|^2+\delta|w_3|^2)>0\right\},
\end{array}\label{domaind4}
\end{equation}
where $\alpha,\beta,\gamma,\delta\ge 0$ and
$$
\det\left(\begin{array}{ll}
\alpha & \beta\\
\gamma & \delta
\end{array}
\right)\ne 0.
$$
As before, we may also assume that $\alpha>0$. Then, if $\beta=\gamma=0$, the domain $D_4$ is biholomorphic to $B^3\times B^2$. In this case $d(D_4)=23=n^2-2$ as desired. Assume now that $\beta+\gamma>0$.

\begin{lemma}\label{g12d4}\it If $\beta+\gamma>0$, for ${\mathfrak g}={\mathfrak g}(D_4)$ one has ${\mathfrak g}_{1/2}=0$.
\end{lemma}
\begin{proof} We will use Theorem \ref{descrg1/2} for the cone $\Omega_1$ and the $\Omega_1$-Hermitian form
\begin{equation}
{\mathcal H}(w,w'):=(\alpha\bar w_1w_1'+\beta\bar w_2w_2'+\beta\bar w_3w_3', \gamma\bar w_1w_1'+\delta\bar w_2w_2'+\delta\bar w_3w_3').\label{hermform2}
\end{equation} 
Let $\Phi:\CC^2\to\CC^3$ be a $\CC$-linear map:
$$
\Phi(z_1,z_2)=(\varphi^1_1z_1+\varphi^1_2z_2,\varphi^2_1z_1+\varphi^2_2z_2,\varphi^3_1z_1+\varphi^3_2z_2),
$$
where $\varphi^i_j\in\CC$. Fixing ${\mathbf w}\in\CC^3$, for $x\in\RR^2$ we compute
$$
\begin{array}{l}
{\mathcal H}({\mathbf w},\Phi(x))=\left(\alpha\bar {\mathbf w}_1(\varphi^1_1x_1+\varphi^1_2x_2)+\beta\bar {\mathbf w}_2(\varphi^2_1x_1+\varphi^2_2x_2)+\beta\bar {\mathbf w}_3(\varphi^3_1x_1+\varphi^3_2x_2),\right.\\
\vspace{-0.3cm}\\
\hspace{2cm}\left.\gamma\bar {\mathbf w}_1(\varphi^1_1x_1+\varphi^1_2x_2)+\delta \bar {\mathbf w}_2(\varphi^2_1x_1+\varphi^2_2x_2)+\delta \bar {\mathbf w}_3(\varphi^3_1x_1+\varphi^3_2x_2)\right)=\\
\vspace{-0.1cm}\\
\left((\alpha\bar {\mathbf w}_1\varphi^1_1+\beta\bar {\mathbf w}_2\varphi^2_1+\beta\bar {\mathbf w}_3\varphi^3_1)x_1+(\alpha\bar {\mathbf w}_1\varphi^1_2+\beta\bar {\mathbf w}_2\varphi^2_2+\beta\bar {\mathbf w}_3\varphi^3_2)x_2,\right.\\
\vspace{-0.1cm}\\
\hspace{2.5cm}\left.(\gamma\bar {\mathbf w}_1\varphi^1_1+\delta \bar {\mathbf w}_2\varphi^2_1+\delta \bar {\mathbf w}_3\varphi^3_1)x_1+(\gamma\bar {\mathbf w}_1\varphi^1_2+\delta \bar {\mathbf w}_2\varphi^2_2+\delta \bar {\mathbf w}_3\varphi^3_2)x_2\right).
\end{array}
$$
Then from formula (\ref{Phiw0}) we obtain
$$
\begin{array}{l}
\Phi_{{\mathbf w}}(x)=\left((\alpha\Im(\bar {\mathbf w}_1\varphi^1_1)+\beta\Im(\bar {\mathbf w}_2\varphi^2_1)+\beta\Im(\bar {\mathbf w}_3\varphi^3_1))x_1+\right.\\
\vspace{-0.3cm}\\
\hspace{2cm}\left.(\alpha\Im(\bar {\mathbf w}_1\varphi^1_2)+\beta\Im(\bar {\mathbf w}_2\varphi^2_2)+\beta\Im(\bar {\mathbf w}_3\varphi^3_2))x_2,\right.\\
\vspace{-0.3cm}\\
\hspace{3cm}\left.(\gamma\Im(\bar {\mathbf w}_1\varphi^1_1)+\delta\Im(\bar {\mathbf w}_2\varphi^2_1)+\delta\Im(\bar {\mathbf w}_3\varphi^3_1))x_1+\right.\\
\vspace{-0.3cm}\\
\hspace{4cm}\left.(\gamma\Im(\bar {\mathbf w}_1\varphi^1_2)+\delta\Im(\bar {\mathbf w}_2\varphi^2_2)+\delta\Im(\bar {\mathbf w}_3\varphi^3_2))x_2\right).
\end{array}
$$
The requirement that this map lies in ${\mathfrak g}(\Omega_1)$ for every ${\mathbf w}\in\CC^3$ is equivalent to
$$
\begin{array}{l}
\alpha\Im(\bar {\mathbf w}_1\varphi^1_2)+\beta\Im(\bar {\mathbf w}_2\varphi^2_2)+\beta\Im(\bar {\mathbf w}_3\varphi^3_2)\equiv 0,\\
\vspace{-0.3cm}\\
\gamma\Im(\bar {\mathbf w}_1\varphi^1_1)+\delta\Im(\bar {\mathbf w}_2\varphi^2_1)+\delta\Im(\bar {\mathbf w}_3\varphi^3_1)\equiv 0,
\end{array}
$$
which leads to the relations
\begin{equation}
\begin{array}{lll}
\varphi^1_2=0, & \beta\varphi^2_2=0, & \beta\varphi^3_2=0,\\
\vspace{-0.3cm}\\
\gamma \varphi^1_1=0, & \delta \varphi^2_1=0, & \delta \varphi^3_1=0.
\end{array}\label{idss11} 
\end{equation}
If each of $\beta$, $\gamma$, $\delta$ is non-zero, it follows from (\ref{idss11}) that $\Phi=0$, which by formula (\ref{cond1}) implies ${\mathfrak g}_{1/2}=0$ as required. 

Assume now that $\beta=0$, hence each of $\gamma$, $\delta$ is non-zero. Then (\ref{idss11}) implies
$$
\varphi^1_1=0,\,\,\varphi^1_2=0,\,\,\varphi^2_1=0,\,\,\varphi^3_1=0.
$$
Hence, $\Phi(z_1,z_2)=(0,\varphi^2_2z_2,\varphi^3_2z_2)$, and for $w,w'\in\CC^3$ we compute
\begin{equation}
\begin{array}{l}
2i{\mathcal H}(\Phi({\mathcal H}(w',w)),w')=\left(0,2i\gamma\delta\bar\varphi^2_2\bar w_1w_1'w_2'+2i\delta^2\bar\varphi^2_2\bar w_2(w_2')^2+\right.\\
\vspace{-0.3cm}\\
\hspace{0.3cm}\left.2i\delta^2\bar\varphi^2_2\bar w_3w_2'w_3'+2i\gamma\delta\bar\varphi^3_2\bar w_1w_1'w_3'+2i\delta^2\bar\varphi^3_2\bar w_2w_2'w_3'+2i\delta^2\bar\varphi^3_2\bar w_3(w_3')^2\right).
\end{array}\label{idss21}
\end{equation}

Further, let $c$ be a symmetric $\CC$-bilinear form on $\CC^3$ with values in $\CC^3$:
\begin{equation}
\begin{array}{l}
c(w,w)=\left(c^1_{11}w_1^2+2c^1_{12}w_1w_2+2c^1_{13}w_1w_3+c^1_{22}w_2^2+2c^1_{23}w_2w_3+c^1_{33}w_3^2,\right.\\
\vspace{-0.3cm}\\
\hspace{1.5cm}\left.c^2_{11}w_1^2+2c^2_{12}w_1w_2+2c^2_{13}w_1w_3+c^2_{22}w_2^2+2c^2_{23}w_2w_3+c^2_{33}w_3^2,\right.\\
\vspace{-0.3cm}\\
\hspace{1.5cm}\left.c^3_{11}w_1^2+2c^3_{12}w_1w_2+2c^3_{13}w_1w_3+c^3_{22}w_2^2+2c^3_{23}w_2w_3+c^3_{33}w_3^2\right),
\end{array}\label{formc}
\end{equation}
where $c^k_{ij}\in\CC$. Then for $w,w'\in\CC^3$ we have
\begin{equation}
\begin{array}{l}
{\mathcal H}(w,c(w',w'))=\left(\alpha\bar w_1(c^1_{11}(w_1')^2+2c^1_{12}w_1'w_2'+2c^1_{13}w_1'w_3'+c^1_{22}(w_2')^2+\right.\\
\vspace{-0.3cm}\\
\hspace{0.3cm}\left.2c^1_{23}w_2'w_3'+c^1_{33}(w_3')^2),\gamma\bar w_1(c^1_{11}(w_1')^2+2c^1_{12}w_1'w_2'+2c^1_{13}w_1'w_3'+\right.\\
\vspace{-0.3cm}\\
\hspace{0.6cm}\left.c^1_{22}(w_2')^2+2c^1_{23}w_2'w_3'+c^1_{33}(w_3')^2)+\delta\bar w_2(c^2_{11}(w_1')^2+2c^2_{12}w_1'w_2'+\right.\\
\vspace{-0.3cm}\\
\hspace{0.8cm}\left.2c^2_{13}w_1'w_3'+c^2_{22}(w_2')^2+2c^2_{23}w_2'w_3'+c^2_{33}(w_3')^2)+
\delta\bar w_3(c^3_{11}(w_1')^2+\right.\\
\vspace{-0.3cm}\\
\hspace{1.3cm}\left.2c^3_{12}w_1'w_2'+2c^3_{13}w_1'w_3'+c^3_{22}(w_2')^2+2c^3_{23}w_2'w_3'+c^3_{33}(w_3')^2)\right).
\end{array}\label{idss31}
\end{equation} 
Comparing the right-hand sides of (\ref{idss21}) and (\ref{idss31}) for arbitrary $w,w'$ as required by (\ref{cond1}), we see that $c^1_{ij}=0$ for all $i,j$ hence $\varphi^2_2=0$, $\varphi^3_2=0$ and therefore $\Phi=0$, $c=0$. Thus for $\beta=0$ we again have ${\mathfrak g}_{1/2}=0$ as claimed.

Suppose next that $\gamma=0$, hence each of $\beta$, $\delta$ is non-zero. In this case (\ref{idss11}) yields
$$
\varphi^1_2=0,\,\,\varphi^2_1=0,\,\,\varphi^2_2=0,\,\,\varphi^3_1=0,\,\,\varphi^3_2=0.
$$
Therefore, $\Phi(z_1,z_2)=(\varphi^1_1z_1,0,0)$, and for $w,w'\in\CC^3$ we find
\begin{equation}
\begin{array}{l}
2i{\mathcal H}(\Phi({\mathcal H}(w',w)),w')=\left(2i\alpha^2\bar\varphi^1_1\bar w_1(w_1')^2+2i\alpha\beta\bar\varphi^1_1\bar w_2w_1'w_2'+\right.\\
\vspace{-0.3cm}\\
\hspace{7.5cm}\left.2i\alpha\beta\bar\varphi^1_1\bar w_3w_1'w_3',0\right).
\end{array}\label{idss22}
\end{equation}
Further, for a symmetric $\CC$-bilinear form on $\CC^3$ with values in $\CC^3$ as in (\ref{formc}) we see
\begin{equation}
\begin{array}{l}
{\mathcal H}(w,c(w',w'))=\left(\alpha\bar w_1(c^1_{11}(w_1')^2+2c^1_{12}w_1'w_2'+2c^1_{13}w_1'w_3'+c^1_{22}(w_2')^2+\right.\\
\vspace{-0.3cm}\\
\hspace{0.3cm}\left.2c^1_{23}w_2'w_3'+c^1_{33}(w_3')^2)+\beta\bar w_2(c^2_{11}(w_1')^2+2c^2_{12}w_1'w_2'+2c^2_{13}w_1'w_3'+\right.\\
\vspace{-0.3cm}\\
\hspace{0.6cm}\left.c^2_{22}(w_2')^2+2c^2_{23}w_2'w_3'+c^2_{33}(w_3')^2)+\beta\bar w_3(c^3_{11}(w_1')^2+2c^3_{12}w_1'w_2'+\right.\\
\vspace{-0.3cm}\\
\hspace{0.9cm}\left.2c^3_{13}w_1'w_3'+c^3_{22}(w_2')^2+2c^3_{23}w_2'w_3'+c^3_{33}(w_3')^2),\delta\bar w_2(c^2_{11}(w_1')^2+\right.\\
\vspace{-0.3cm}\\
\hspace{1.2cm}\left.2c^2_{12}w_1'w_2'+2c^2_{13}w_1'w_3'+c^2_{22}(w_2')^2+2c^2_{23}w_2'w_3'+c^2_{33}(w_3')^2)+\right.\\
\vspace{-0.3cm}\\
\hspace{1.5cm}\left.\delta\bar w_3(c^3_{11}(w_1')^2+2c^3_{12}w_1'w_2'+2c^3_{13}w_1'w_3'+c^3_{22}(w_2')^2+\right.\\
\vspace{-0.3cm}\\
\hspace{2cm}\left.2c^3_{23}w_2'w_3'+c^3_{33}(w_3')^2)\right).
\end{array}\label{idss32}
\end{equation} 
Comparing the right-hand sides of (\ref{idss22}) and (\ref{idss32}) for arbitrary $w,w'$, we observe that $c^2_{ij}=0$, $c^3_{ij}=0$ for all $i,j$ hence $\varphi^1_1=0$, and therefore $\Phi=0$, $c=0$. Thus for $\gamma=0$ we see that ${\mathfrak g}_{1/2}=0$ as well.

The case $\delta=0$ is obtained from the case $\beta=0$ by permutation of variables.\end{proof}

By estimate (\ref{estim 8}), the second inequality in (\ref{estimm}), and Lemma \ref{g12d4}, for $\beta+\gamma>0$ we see
\begin{equation}
d(D_4)\le 17<23=n^2-2\label{essstimm11} 
\end{equation}
(recall that here $s=5$). This shows that $S(\Omega,H)$ cannot be equivalent to $D_4$ unless $\beta=\gamma=0$, so Case (1b) only contributes the product $B^3\times B^2$ to the classification of homogeneous Kobayashi-hyperbolic $n$-dimensional manifolds with automorphism group dimension $n^2-2$.

\begin{remark}\label{remg1d4}
In Proposition \ref{g1d4} in the appendix we prove that for $\beta+\gamma>0$ the component ${\mathfrak g}_1$ of the algebra ${\mathfrak g}={\mathfrak g}(D_4)$ is also zero, which improves estimate (\ref{essstimm11}) to $d(D_4)\le 15$.
\end{remark}
\vspace{0.1cm}

{\bf Case (2).} Suppose that $k=3$, $n=4$. Here $S(\Omega,H)$ is linearly equivalent either to
\begin{equation}
D_5:=\left\{(z,w)\in\times\CC^3\times\CC:\Im z-v|w|^2\in\Omega_2\right\},\label{domaind5}
\end{equation}
where $v=(v_1,v_2,v_3)$ is a non-zero vector in $\RR^3$ with non-negative entries, or to
\begin{equation}
D_6:=\left\{(z,w)\in\times\CC^3\times\CC:\Im z-v|w|^2\in\Omega_3\right\},\label{domaind6}
\end{equation}
where $v=(v_1,v_2,v_3)$ is a vector in $\RR^3$ satisfying $v_1^2\ge v_2^2+v_3^2$, $v_1>0$. We will consider these two cases separately.
\vspace{0.1cm}

{\bf Case (2a).} Assume that $S(\Omega,H)$ is equivalent to the domain $D_5$ defined in (\ref{domaind5}). If only one entry of $v$ is non-zero, $D_5$ is biholomorphic to $B^2\times B^1\times B^1$. Notice that $d(B^2\times B^1\times B^1)=14=n^2-2$ as desired. 

Suppose now that at least two entries of $v$ are non-zero and consider the identity component $G(\Omega_2,v|w|^2)^{\circ}$ of the group $G(\Omega_2,v|w|^2)$. As $G(\Omega_2,v|w|^2)^{\circ}$ lies in the identity component $G(\Omega_2)^{\circ}$ of $G(\Omega_2)$, every element of $G(\Omega_2,v|w|^2)^{\circ}$ is a diagonal matrix
\begin{equation}
\left(\begin{array}{lll}
\lambda_1 & 0 & 0\\
0 & \lambda_2 & 0\\
0 & 0 & \lambda_3
\end{array}
\right),\quad\lambda_j>0,\,\,j=1,2,3,\label{diagmatrrrcces}
\end{equation}
for which $v$ is an eigenvector. Therefore, if all entries of $v$ are non-zero, then $G(\Omega_2,v|w|^2)^{\circ}$ consists of scalar matrices, and if exactly two entries of $v$, say $v_i$ and $v_j$, are non-zero, then $G(\Omega_2,v|w|^2)^{\circ}$ consists of matrices of the form (\ref{diagmatrrrcces}) with $\lambda_i=\lambda_j$. In either situation, the action of $G(\Omega_2,v|w|^2)^{\circ}$ on $\Omega_2$ is not transitive. This shows that $S(\Omega,H)$ cannot be equivalent to $D_5$ unless exactly one entry of $v$ is non-zero, so Case (2a) only contributes the product $B^2\times B^1\times B^1$ to the classification of homogeneous Kobayashi-hyperbolic $n$-dimensional manifolds with automorphism group dimension $n^2-2$.
\vspace{0.1cm}

{\bf Case (2b).} Assume now that $S(\Omega,H)$ is equivalent to the domain $D_6$ defined in (\ref{domaind6}). Suppose first that $v_1^2>v_2^2+v_3^2$, i.e., that $v\in\Omega_3$. As the vector $v$ is an eigenvector of every element of $G(\Omega_3,v|w|^2)$, it then follows that $G(\Omega_3,v|w|^2)$ does not act transitively on $\Omega_3$. This shows that in fact we have $v_1=\sqrt{v_2^2+v_3^2}\ne 0$, i.e., $v\in\partial\Omega_3\setminus\{0\}$. Further, as the group $G(\Omega_3)^{\circ}=\RR_{+}\times\SO(1,2)^{\circ}$ acts transitively on $\partial\Omega_3\setminus\{0\}$, we suppose from now on that $v=(1,1,0)$.

\begin{lemma}\label{g12d6}\it For ${\mathfrak g}={\mathfrak g}(D_3)$ one has ${\mathfrak g}_{1/2}=0$.
\end{lemma}

\begin{proof}We will apply Theorem \ref{descrg1/2} to the cone $\Omega_3$ and the $\Omega_3$-Hermitian form
\begin{equation}
{\mathcal H}(w,w'):=(\bar ww',\bar ww',0).\label{formh444}
\end{equation}
Let $\Phi:\CC^3\to\CC$ be a $\CC$-linear map:
$$
\Phi(z_1,z_2,z_3)=\varphi_1z_1+\varphi_2z_2+\varphi_3z_3,
$$
where $\varphi_j\in\CC$. Fixing ${\mathbf w}\in\CC$, for $x\in\RR^3$ we compute
$$
{\mathcal H}({\mathbf w},\Phi(x))=\left(\bar {\mathbf w}(\varphi_1x_1+\varphi_2x_2+\varphi_3x_3),\bar {\mathbf w}(\varphi_1x_1+\varphi_2x_2+\varphi_3x_3),0\right).
$$
Then from formula (\ref{Phiw0}) we see
$$
\begin{array}{l}
\Phi_{{\mathbf w}}(x)=\left(\Im(\bar {\mathbf w}\varphi_1)x_1+\Im(\bar {\mathbf w}\varphi_2)x_2+\Im(\bar {\mathbf w}\varphi_3)x_3,\right.\\
\vspace{-0.3cm}\\
\hspace{5cm}\left.\Im(\bar {\mathbf w}\varphi_1)x_1+\Im(\bar {\mathbf w}\varphi_2)x_2+\Im(\bar {\mathbf w}\varphi_3)x_3,0\right).
\end{array}
$$

Now, recall that ${\mathfrak g}(\Omega_3)={\mathfrak c}({\mathfrak{gl}}_3(\RR))\oplus{\mathfrak o}_{1,2}$ consists of all matrices of the form
\begin{equation}
\left(\begin{array}{lll}
\lambda & p & q\\
p & \lambda & r\\
q & -r & \lambda
\end{array}
\right),\quad \lambda,p,q,r\in\RR.\label{algomega3}
\end{equation}
Therefore, the condition that the map $\Phi_{{\mathbf w}}$ lies in ${\mathfrak g}(\Omega_3)$ for every ${\mathbf w}\in\CC$ immediately yields
$$
\Im(\bar {\mathbf w}\varphi_1)\equiv 0,\,\,\Im(\bar {\mathbf w}\varphi_2)\equiv 0,\,\,\Im(\bar {\mathbf w}\varphi_3)\equiv 0,
$$
which implies $\Phi=0$. Hence, by formula (\ref{cond1}) we have ${\mathfrak g}_{1/2}=0$ as required.\end{proof}

By estimate (\ref{estim 8}), the second inequality in (\ref{estimm}), and Lemma \ref{g12d6}, we see
\begin{equation}
d(D_6)\le 13<14=n^2-2\label{essstimm111} 
\end{equation}
(notice that here $s=1$). This shows that $S(\Omega,H)$ cannot be equivalent to $D_6$, so Case (2b) contributes nothing to the classification of homogeneous Kobayashi-hyperbolic $n$-dimensional manifolds with automorphism group dimension $n^2-2$.

\begin{remark}\label{remg1d6}
In Proposition \ref{g1d6} in the appendix we prove that the component ${\mathfrak g}_1$ of the algebra ${\mathfrak g}={\mathfrak g}(D_6)$ is 1-dimensional, which improves estimate (\ref{essstimm111}) to $d(D_6)\le 11$. In fact, it is not hard to see that for the form ${\mathcal H}$ introduced in (\ref{formh444}) one has $\dim G(\Omega_3,{\mathcal H})=3$, and therefore $\dim {\mathfrak g}_0=4$. It then follows that
$$
d(D_6)=\dim{\mathfrak g}_{-1}+\dim{\mathfrak g}_{-1/2}+\dim{\mathfrak g}_0+\dim{\mathfrak g}_1=10.
$$
\end{remark}
\vspace{0.1cm}

{\bf Case (3).} Suppose that $k=4$, $n=4$. In this case, after a linear change of variables $S(\Omega,H)$ turns into one of the domains
$$
\begin{array}{l}
\left\{z\in\CC^4: \Im z\in\Omega_4\right\},\\
\vspace{-0.1cm}\\
\left\{z\in\CC^4: \Im z\in\Omega_5\right\},\\
\vspace{-0.1cm}\\
\left\{z\in\CC^4: \Im z\in\Omega_6\right\}
\end{array}
$$
and therefore is biholomorphic either to $B^1\times B^1\times B^1\times B^1$, or to $B^1\times T_3$, or to $T_4$, where $T_3$ and $T_4$ are the tube domains defined in (\ref{domaint3}), (\ref{domaint4}). The dimensions of the automorphism groups of these domains are 12, 13, 15, respectively. As none of the numbers is equal to $14=n^2-2$, Case (3) contributes nothing to the classification of homogeneous Kobayashi-hyperbolic $n$-dimensional manifolds with automorphism group dimension $n^2-2$. 

The proof of Theorem \ref{main} is now complete.\qed

\section*{Appendix}\label{appendix}
\setcounter{equation}{0}
\renewcommand{\theequation}{A.\arabic{equation}} 
\renewcommand{\thesubsection}{A.\arabic{subsection}}
\renewcommand{\thetheorem}{A.\arabic{theorem}}

Here we show that for the domains $D_3$, $D_4$ introduced in (\ref{domaind3}), (\ref{domaind4}), respectively, we have ${\mathfrak g}_1=0$ if $\beta+\gamma>0$. In addition, we prove that for the domain $D_6$ defined in (\ref{domaind6}) one has $\dim{\mathfrak g}_1=1$ if $v\in\partial\Omega_3\setminus\{0\}$. These facts can be utilized for extending our classifications in Theorems \ref{main}, \ref{combined} to automorphism group dimensions less than the critical dimension $n^2-2$. The proofs below are also independently interesting as they contain explicit computations with the fairly bulky formulas supplied by Theorem \ref{descrg1}, which is rarely seen in the literature.

We start with the domain $D_3$.

\begin{proposition}\label{g1d3}\it If $\beta+\gamma>0$, for ${\mathfrak g}={\mathfrak g}(D_3)$ one has ${\mathfrak g}_1=0$.
\end{proposition}

\begin{proof} We will utilize Theorem \ref{descrg1} for the cone $\Omega_1$ and the $\Omega_1$-Hermitian form ${\mathcal H}$ given by (\ref{hermform1}). Consider a symmetric $\RR$-bilinear form on $\RR^2$ with values in $\RR^2$:
\begin{equation}
a(x,x)=\left(a_{11}^1x_1^2+2a_{12}^1x_1x_2+a_{22}^1x_2^2,a_{11}^2x_1^2+2a_{12}^2x_1x_2+a_{22}^2x_2^2\right),\label{symmforma}
\end{equation}
where $a_{ij}^k\in\RR$. Then for a fixed ${\mathbf x}\in\RR^2$ from (\ref{idents1}) we compute
$$
\begin{array}{l} 
A_{{\mathbf x}}(x)=\left(a_{11}^1{\mathbf x}_1x_1+a_{12}^1{\mathbf x}_1x_2+a_{12}^1{\mathbf x}_2x_1+a_{22}^1{\mathbf x}_2x_2,\right.\\
\vspace{-0.3cm}\\
\hspace{3cm}\left.a_{11}^2{\mathbf x}_1x_1+a_{12}^2{\mathbf x}_1x_2+a_{12}^2{\mathbf x}_2x_1+a_{22}^2{\mathbf x}_2x_2\right)=\\
\vspace{-0.3cm}\\
\left((a_{11}^1{\mathbf x}_1+a_{12}^1{\mathbf x}_2)x_1+(a_{12}^1{\mathbf x}_1+a_{22}^1{\mathbf x}_2)x_2,(a_{11}^2{\mathbf x}_1+a_{12}^2{\mathbf x}_2)x_1+(a_{12}^2{\mathbf x}_1+a_{22}^2{\mathbf x}_2)x_2\right),
\end{array}
$$
where $x\in\RR^2$. The condition that this map lies in ${\mathfrak g}(\Omega_1)$ for every ${\mathbf x}\in\RR^2$ is equivalent to
$$
\begin{array}{l}
a_{12}^1{\mathbf x}_1+a_{22}^1{\mathbf x}_2\equiv 0,\\
\vspace{-0.3cm}\\
a_{11}^2{\mathbf x}_1+a_{12}^2{\mathbf x}_2\equiv 0,
\end{array}
$$  
which implies
\begin{equation}
a_{12}^1=0,\quad a_{22}^1=0,\quad a_{11}^2=0,\quad a_{12}^2=0.\label{relcoeffa}
\end{equation}
Therefore,
\begin{equation}
A_{{\mathbf x}}(x)=(a_{11}^1{\mathbf x}_1x_1, a_{22}^2{\mathbf x}_2x_2).\label{forma}
\end{equation}

Next, let $b:\CC^2\times\CC^2\to\CC^2$ be a $\CC$-bilinear map:
$$
\begin{array}{l}
b(z,w)=\left(b_{11}^1z_1w_1+b_{12}^1z_1w_2+b_{21}^1z_2w_1+b_{22}^1z_2w_2,\right.\\
\vspace{-0.3cm}\\
\hspace{5cm}\left.b_{11}^2z_1w_1+b_{12}^2z_1w_2+b_{21}^2z_2w_1+b_{22}^2z_2w_2\right),
\end{array}
$$
where $b_{ij}^k\in\CC$. Then for a fixed ${\mathbf x}\in\RR^2$ from (\ref{idents2}) we find 
$$
\begin{array}{l}
\displaystyle B_{{\mathbf x}}(w)=\displaystyle\frac{1}{2}\left((b_{11}^1{\mathbf x}_1+b_{21}^1{\mathbf x}_2)w_1+(b_{12}^1{\mathbf x}_1+b_{22}^1{\mathbf x}_2)w_2,\right.\\
\vspace{-0.3cm}\\
\hspace{5cm}\left.(b_{11}^2{\mathbf x}_1+b_{21}^2{\mathbf x}_2)w_1+(b_{12}^2{\mathbf x}_1+b_{22}^2{\mathbf x}_2)w_2\right).
\end{array}
$$
The condition that $\Im\tr B_{{\mathbf x}}=0$ for all ${\mathbf x}\in\RR^2$ in (i) in Theorem \ref{descrg1} means
$$
\Im((b_{11}^1+b_{12}^2){\mathbf x}_1+(b_{21}^1+b_{22}^2){\mathbf x}_2)\equiv 0,
$$
which leads, in particular, to
\begin{equation}
\Im(b_{21}^1+b_{22}^2)=0.\label{imrel1}
\end{equation}

Further, for every fixed pair ${\mathbf w},{\mathbf w}'\in\CC^2$ we compute
$$
\begin{array}{l}
{\mathcal H}({\mathbf w}',b(x,{\mathbf w}))=\left(\alpha\bar{\mathbf w}_1'(b_{11}^1{\mathbf w}_1x_1+b_{12}^1{\mathbf w}_2x_1+b_{21}^1{\mathbf w}_1x_2+b_{22}^1{\mathbf w}_2x_2)+\right.\\
\vspace{-0.3cm}\\
\hspace{0.5cm}\left.\beta\bar{\mathbf w}_2'(b_{11}^2{\mathbf w}_1x_1+b_{12}^2{\mathbf w}_2x_1+b_{21}^2{\mathbf w}_1x_2+b_{22}^2{\mathbf w}_2x_2),\right.\\
\vspace{-0.3cm}\\
\hspace{1cm}\left.\gamma\bar{\mathbf w}_1'(b_{11}^1{\mathbf w}_1x_1+b_{12}^1{\mathbf w}_2x_1+b_{21}^1{\mathbf w}_1x_2+b_{22}^1{\mathbf w}_2x_2)+\right.\\
\vspace{-0.3cm}\\
\hspace{1.5cm}\left.\delta\bar{\mathbf w}_2'(b_{11}^2{\mathbf w}_1x_1+b_{12}^2{\mathbf w}_2x_1+b_{21}^2{\mathbf w}_1x_2+b_{22}^2{\mathbf w}_2x_2)\right)=\\
\vspace{-0.3cm}\\
\hspace{2cm}\left((\alpha\bar{\mathbf w}_1'(b_{11}^1{\mathbf w}_1+b_{12}^1{\mathbf w}_2)+\beta\bar{\mathbf w}_2'(b_{11}^2{\mathbf w}_1+b_{12}^2{\mathbf w}_2))x_1+\right.\\
\vspace{-0.3cm}\\
\hspace{2.5cm}\left.(\alpha\bar{\mathbf w}_1'(b_{21}^1{\mathbf w}_1+b_{22}^1{\mathbf w}_2)+\beta\bar{\mathbf w}_2'(b_{21}^2{\mathbf w}_1+b_{22}^2{\mathbf w}_2))x_2,\right.\\
\vspace{-0.3cm}\\
\hspace{3cm}\left.(\gamma\bar{\mathbf w}_1'(b_{11}^1{\mathbf w}_1+b_{12}^1{\mathbf w}_2)+\delta\bar{\mathbf w}_2'(b_{11}^2{\mathbf w}_1+b_{12}^2{\mathbf w}_2))x_1+\right.\\
\vspace{-0.3cm}\\
\hspace{3.5cm}\left.(\gamma\bar{\mathbf w}_1'(b_{21}^1{\mathbf w}_1+b_{22}^1{\mathbf w}_2)+\delta\bar{\mathbf w}_2'(b_{21}^2{\mathbf w}_1+b_{22}^2{\mathbf w}_2))x_2\right).
\end{array}
$$
Then from (ii) of Theorem \ref{descrg1} we obtain
$$
\begin{array}{l}
B_{{\mathbf w},{\mathbf w}'}(x)=\left(\Im(\alpha\bar{\mathbf w}_1'(b_{11}^1{\mathbf w}_1+b_{12}^1{\mathbf w}_2)+\beta\bar{\mathbf w}_2'(b_{11}^2{\mathbf w}_1+b_{12}^2{\mathbf w}_2))x_1+\right.\\
\vspace{-0.3cm}\\
\hspace{2.5cm}\left.\Im(\alpha\bar{\mathbf w}_1'(b_{21}^1{\mathbf w}_1+b_{22}^1{\mathbf w}_2)+\beta\bar{\mathbf w}_2'(b_{21}^2{\mathbf w}_1+b_{22}^2{\mathbf w}_2))x_2,\right.\\
\vspace{-0.3cm}\\
\hspace{3cm}\left.\Im(\gamma\bar{\mathbf w}_1'(b_{11}^1{\mathbf w}_1+b_{12}^1{\mathbf w}_2)+\delta\bar{\mathbf w}_2'(b_{11}^2{\mathbf w}_1+b_{12}^2{\mathbf w}_2))x_1+\right.\\
\vspace{-0.3cm}\\
\hspace{3.5cm}\left.\Im(\gamma\bar{\mathbf w}_1'(b_{21}^1{\mathbf w}_1+b_{22}^1{\mathbf w}_2)+\delta\bar{\mathbf w}_2'(b_{21}^2{\mathbf w}_1+b_{22}^2{\mathbf w}_2))x_2\right).
\end{array}
$$
The condition that this map lies in ${\mathfrak g}(\Omega_1)$ for all ${\mathbf w},{\mathbf w}'\in\CC^2$ means
$$
\begin{array}{l}
\Im(\alpha\bar{\mathbf w}_1'(b_{21}^1{\mathbf w}_1+b_{22}^1{\mathbf w}_2)+\beta\bar{\mathbf w}_2'(b_{21}^2{\mathbf w}_1+b_{22}^2{\mathbf w}_2))\equiv 0,\\
\vspace{-0.3cm}\\
\Im(\gamma\bar{\mathbf w}_1'(b_{11}^1{\mathbf w}_1+b_{12}^1{\mathbf w}_2)+\delta\bar{\mathbf w}_2'(b_{11}^2{\mathbf w}_1+b_{12}^2{\mathbf w}_2))\equiv 0,
\end{array}
$$
which yields
\begin{equation}
\begin{array}{llll}
b_{21}^1=0, & b_{22}^1=0, & \beta b_{21}^2=0, & \beta b_{22}^2=0,\\
\vspace{-0.3cm}\\
\gamma b_{11}^1=0, & \gamma b_{12}^1=0, & \delta b_{11}^2=0, & \delta b_{12}^2=0.
\end{array}\label{zero1}
\end{equation}
If each of $\beta$, $\gamma$, $\delta$ is non-zero, it follows from (\ref{zero1}) that $b=0$, therefore $B_{{\mathbf x}}=0$ for all ${\mathbf x}\in\RR^2$, and the requirement that $B_{{\mathbf x}}$ is associated to $A_{{\mathbf x}}$ with respect to ${\mathcal H}$ for every ${\mathbf x}\in\RR^2$ (see condition (i) in Theorem \ref{descrg1}) implies $a=0$. Thus, ${\mathfrak g}_1=0$ as claimed.

Assume now that $\beta=0$, hence each of $\gamma$, $\delta$ is non-zero, so by (\ref{zero1}) we have
$$
b_{11}^1=0,\,\,b_{12}^1=0,\,\,b_{21}^1=0, \,\,b_{22}^1=0,\,\, b_{11}^2=0, \,\, b_{12}^2=0,
$$
and (\ref{imrel1}) yields
\begin{equation}
\Im b_{22}^2=0.\label{imrel2}
\end{equation}
Thus,
\begin{equation}
B_{{\mathbf x}}(w)=\displaystyle\frac{1}{2}(0,b_{21}^2{\mathbf x}_2w_1+b_{22}^2{\mathbf x}_2w_2).\label{formb}
\end{equation}

We will now use the requirement that $B_{{\mathbf x}}$ is associated to $A_{{\mathbf x}}$ with respect to ${\mathcal H}$ for every ${\mathbf x}\in\RR^2$ as in (i) in Theorem \ref{descrg1}. On the one hand, from (\ref{forma}) we have
\begin{equation}
A_{{\mathbf x}}{\mathcal H}(w,w')=(a_{11}^1{\mathbf x}_1\alpha\bar w_1w_1',a_{22}^2{\mathbf x}_2(\gamma \bar w_1w_1'+\delta \bar w_2w_2')).\label{left1}
\end{equation}
On the other hand, from (\ref{formb}) one obtains
\begin{equation}
\begin{array}{l}
\displaystyle{\mathcal H}(B_{{\mathbf x}}(w),w')+{\mathcal H}(w,B_{{\mathbf x}}(w'))=\frac{1}{2}\left(0,\delta(\bar b_{21}^2{\mathbf x}_2\bar w_1+\bar b_{22}^2{\mathbf x}_2\bar w_2)w_2'+\right.\\
\vspace{-0.3cm}\\
\hspace{7cm}\left.\delta\bar w_2(b_{21}^2{\mathbf x}_2w_1'+b_{22}^2{\mathbf x}_2w_2')\right).
 \end{array}\label{right1}
\end{equation}
Comparing (\ref{left1}) and  (\ref{right1}), we deduce
\begin{equation}
a_{11}^1=0,\,\,a_{22}^2=0,\,\,b_{21}^2=0,\,\, \Re b_{22}^2=0.\label{zero4}
\end{equation}
By (\ref{imrel2}) and (\ref{zero4}) we see that $a=0$, $b=0$, which shows that ${\mathfrak g}_1=0$ as claimed.

The cases $\gamma=0$ and $\delta=0$ are obtained from the case $\beta=0$ by permutation of variables.\end{proof}

Next, we will deduce an analogous fact for the domain $D_4$.

\begin{proposition}\label{g1d4}\it If $\beta+\gamma>0$, for ${\mathfrak g}={\mathfrak g}(D_4)$ one has ${\mathfrak g}_1=0$.
\end{proposition}

\begin{proof} We will apply Theorem \ref{descrg1} to the cone $\Omega_1$ and $\Omega_1$-Hermitian form (\ref{hermform2}). Consider a symmetric $\RR$-bilinear form on $\RR^2$ with values in $\RR^2$ (see (\ref{symmforma})). The proof of Proposition \ref{g1d3} shows that the coefficients of this form satisfy (\ref{relcoeffa}). Hence, for every fixed ${\mathbf x}\in\RR^2$ the map $A_{{\mathbf x}}$ defined in (\ref{idents1}) is given by formula (\ref{forma}).

Next, let $b:\CC^2\times\CC^3\to\CC^3$ be a $\CC$-bilinear map:
$$
\begin{array}{l}
b(z,w)=\left(b_{11}^1z_1w_1+b_{12}^1z_1w_2+b_{13}^1z_1w_3+b_{21}^1z_2w_1+b_{22}^1z_2w_2+b_{23}^1z_2w_3,\right.\\
\vspace{-0.3cm}\\
\hspace{1.6cm}\left.b_{11}^2z_1w_1+b_{12}^2z_1w_2+b_{13}^2z_1w_3+b_{21}^2z_2w_1+b_{22}^2z_2w_2+b_{23}^2z_2w_3,\right.\\
\vspace{-0.3cm}\\
\hspace{1.6cm}\left.b_{11}^3z_1w_1+b_{12}^3z_1w_2+b_{13}^3z_1w_3+b_{21}^3z_2w_1+b_{22}^3z_2w_2+b_{23}^3z_2w_3\right),
\end{array}
$$
where $b_{ij}^k\in\CC$. Then for a fixed ${\mathbf x}\in\RR^2$ from (\ref{idents2}) we compute 
$$
\begin{array}{l}
\displaystyle B_{{\mathbf x}}(w)=\displaystyle\frac{1}{2}\left((b_{11}^1{\mathbf x}_1+b_{21}^1{\mathbf x}_2)w_1+(b_{12}^1{\mathbf x}_1+b_{22}^1{\mathbf x}_2)w_2+(b_{13}^1{\mathbf x}_1+b_{23}^1{\mathbf x}_2)w_3,\right.\\
\vspace{-0.3cm}\\
\hspace{1.9cm}\left.(b_{11}^2{\mathbf x}_1+b_{21}^2{\mathbf x}_2)w_1+(b_{12}^2{\mathbf x}_1+b_{22}^2{\mathbf x}_2)w_2+(b_{13}^2{\mathbf x}_1+b_{23}^2{\mathbf x}_2)w_3,\right.\\
\vspace{-0.3cm}\\
\hspace{1.9cm}\left.(b_{11}^3{\mathbf x}_1+b_{21}^3{\mathbf x}_2)w_1+(b_{12}^3{\mathbf x}_1+b_{22}^3{\mathbf x}_2)w_2+(b_{13}^3{\mathbf x}_1+b_{23}^3{\mathbf x}_2)w_3\right).
\end{array}
$$
The condition that $\Im\tr B_{{\mathbf x}}=0$ for all ${\mathbf x}\in\RR^2$ in (i) in Theorem \ref{descrg1} means
$$
\Im((b_{11}^1+b_{12}^2+b_{13}^3){\mathbf x}_1+(b_{21}^1+b_{22}^2+b_{23}^3){\mathbf x}_2)\equiv 0,
$$
which leads, in particular, to
\begin{equation}
\Im(b_{11}^1+b_{12}^2+b_{13}^3)=0.\label{imrel11}
\end{equation}

Further, for every fixed pair ${\mathbf w},{\mathbf w}'\in\CC^3$ we find
$$
\begin{array}{l}
{\mathcal H}({\mathbf w}',b(x,{\mathbf w}))=\left(\alpha\bar{\mathbf w}_1'(b_{11}^1{\mathbf w}_1x_1+b_{12}^1{\mathbf w}_2x_1+b_{13}^1{\mathbf w}_3x_1+b_{21}^1{\mathbf w}_1x_2+b_{22}^1{\mathbf w}_2x_2+\right.\\
\vspace{-0.3cm}\\
\hspace{0.1cm}\left.b_{23}^1{\mathbf w}_3x_2)+\beta\bar{\mathbf w}_2'(b_{11}^2{\mathbf w}_1x_1+b_{12}^2{\mathbf w}_2x_1+b_{13}^2{\mathbf w}_3x_1+b_{21}^2{\mathbf w}_1x_2+b_{22}^2{\mathbf w}_2x_2+\right.\\
\vspace{-0.3cm}\\
\hspace{0.3cm}\left.b_{23}^2{\mathbf w}_3x_2)+\beta\bar{\mathbf w}_3'(b_{11}^3{\mathbf w}_1x_1+b_{12}^3{\mathbf w}_2x_1+b_{13}^3{\mathbf w}_3x_1+b_{21}^3{\mathbf w}_1x_2+b_{22}^3{\mathbf w}_2x_2+\right.\\
\vspace{-0.3cm}\\
\hspace{0.5cm}\left.b_{23}^3{\mathbf w}_3x_2),\gamma\bar{\mathbf w}_1'(b_{11}^1{\mathbf w}_1x_1+b_{12}^1{\mathbf w}_2x_1+b_{13}^1{\mathbf w}_3x_1+b_{21}^1{\mathbf w}_1x_2+b_{22}^1{\mathbf w}_2x_2+\right.\\
\vspace{-0.3cm}\\
\hspace{0.7cm}\left.b_{23}^1{\mathbf w}_3x_2)+\delta\bar{\mathbf w}_2'(b_{11}^2{\mathbf w}_1x_1+b_{12}^2{\mathbf w}_2x_1+b_{13}^2{\mathbf w}_3x_1+b_{21}^2{\mathbf w}_1x_2+b_{22}^2{\mathbf w}_2x_2+\right.\\
\vspace{-0.3cm}\\
\hspace{0.9cm}\left.b_{23}^2{\mathbf w}_3x_2)+\delta\bar{\mathbf w}_3'(b_{11}^3{\mathbf w}_1x_1+b_{12}^3{\mathbf w}_2x_1+b_{13}^3{\mathbf w}_3x_1+b_{21}^3{\mathbf w}_1x_2+b_{22}^3{\mathbf w}_2x_2+\right.\\
\vspace{-0.3cm}\\
\hspace{1.1cm}\left.b_{23}^3{\mathbf w}_3x_2)\right)=\left((\alpha\bar{\mathbf w}_1'(b_{11}^1{\mathbf w}_1+b_{12}^1{\mathbf w}_2+b_{13}^1{\mathbf w}_3)+\beta\bar{\mathbf w}_2'(b_{11}^2{\mathbf w}_1+b_{12}^2{\mathbf w}_2+\right.\\
\vspace{-0.3cm}\\
\hspace{1.3cm}\left.b_{13}^2{\mathbf w}_3)+\beta\bar{\mathbf w}_3'(b_{11}^3{\mathbf w}_1+b_{12}^3{\mathbf w}_2+b_{13}^3{\mathbf w}_3))x_1+(\alpha\bar{\mathbf w}_1'(b_{21}^1{\mathbf w}_1+b_{22}^1{\mathbf w}_2+\right.\\
\vspace{-0.3cm}\\
\hspace{1.5cm}\left.b_{23}^1{\mathbf w}_3)+\beta\bar{\mathbf w}_2'(b_{21}^2{\mathbf w}_1+b_{22}^2{\mathbf w}_2+b_{23}^2{\mathbf w}_3)+\beta\bar{\mathbf w}_3'(b_{21}^3{\mathbf w}_1+b_{22}^3{\mathbf w}_2+\right.\\
\vspace{-0.3cm}\\
\hspace{1.7cm}\left.b_{23}^3{\mathbf w}_3))x_2,(\gamma\bar{\mathbf w}_1'(b_{11}^1{\mathbf w}_1+b_{12}^1{\mathbf w}_2+b_{13}^1{\mathbf w}_3)+\delta\bar{\mathbf w}_2'(b_{11}^2{\mathbf w}_1+b_{12}^2{\mathbf w}_2+\right.\\
\vspace{-0.3cm}\\
\hspace{1.9cm}\left.b_{13}^2{\mathbf w}_3)+\delta\bar{\mathbf w}_3'(b_{11}^3{\mathbf w}_1+b_{12}^3{\mathbf w}_2+b_{13}^3{\mathbf w}_3))x_1+(\gamma\bar{\mathbf w}_1'(b_{21}^1{\mathbf w}_1+b_{22}^1{\mathbf w}_2+\right.\\
\vspace{-0.3cm}\\
\hspace{2.1cm}\left.b_{23}^1{\mathbf w}_3)+\delta\bar{\mathbf w}_2'(b_{21}^2{\mathbf w}_1+b_{22}^2{\mathbf w}_2+b_{23}^2{\mathbf w}_3)+\delta\bar{\mathbf w}_3'(b_{21}^3{\mathbf w}_1+b_{22}^3{\mathbf w}_2+\right.\\
\vspace{-0.3cm}\\
\hspace{2.5cm}\left.b_{23}^3{\mathbf w}_3))x_2\right).
\end{array}
$$
Then from (ii) of Theorem \ref{descrg1} we see
$$
\begin{array}{l}
B_{{\mathbf w},{\mathbf w}'}(x)=\left(\Im(\alpha\bar{\mathbf w}_1'(b_{11}^1{\mathbf w}_1+b_{12}^1{\mathbf w}_2+b_{13}^1{\mathbf w}_3)+\beta\bar{\mathbf w}_2'(b_{11}^2{\mathbf w}_1+b_{12}^2{\mathbf w}_2+b_{13}^2{\mathbf w}_3)+\right.\\
\vspace{-0.3cm}\\
\hspace{0.3cm}\left.\beta\bar{\mathbf w}_3'(b_{11}^3{\mathbf w}_1+b_{12}^3{\mathbf w}_2+b_{13}^3{\mathbf w}_3))x_1+\Im(\alpha\bar{\mathbf w}_1'(b_{21}^1{\mathbf w}_1+b_{22}^1{\mathbf w}_2+b_{23}^1{\mathbf w}_3)+\right.\\
\vspace{-0.3cm}\\
\hspace{0.5cm}\left.\beta\bar{\mathbf w}_2'(b_{21}^2{\mathbf w}_1+b_{22}^2{\mathbf w}_2+b_{23}^2{\mathbf w}_3)+\beta\bar{\mathbf w}_3'(b_{21}^3{\mathbf w}_1+b_{22}^3{\mathbf w}_2+b_{23}^3{\mathbf w}_3))x_2,\right.\\
\vspace{-0.3cm}\\
\hspace{0.7cm}\left.\Im(\gamma\bar{\mathbf w}_1'(b_{11}^1{\mathbf w}_1+b_{12}^1{\mathbf w}_2+b_{13}^1{\mathbf w}_3)+\delta\bar{\mathbf w}_2'(b_{11}^2{\mathbf w}_1+b_{12}^2{\mathbf w}_2+b_{13}^2{\mathbf w}_3)+\right.\\
\vspace{-0.3cm}\\
\hspace{0.9cm}\left.\delta\bar{\mathbf w}_3'(b_{11}^3{\mathbf w}_1+b_{12}^3{\mathbf w}_2+b_{13}^3{\mathbf w}_3))x_1+\Im(\gamma\bar{\mathbf w}_1'(b_{21}^1{\mathbf w}_1+b_{22}^1{\mathbf w}_2+b_{23}^1{\mathbf w}_3)+\right.\\
\vspace{-0.3cm}\\
\hspace{1.1cm}\left.\delta\bar{\mathbf w}_2'(b_{21}^2{\mathbf w}_1+b_{22}^2{\mathbf w}_2+b_{23}^2{\mathbf w}_3)+\delta\bar{\mathbf w}_3'(b_{21}^3{\mathbf w}_1+b_{22}^3{\mathbf w}_2+b_{23}^3{\mathbf w}_3))x_2\right).
\end{array}
$$
The condition that this map lies in ${\mathfrak g}(\Omega_1)$ for all ${\mathbf w},{\mathbf w}'\in\CC^3$ is equivalent to
$$
\begin{array}{l}
\Im(\alpha\bar{\mathbf w}_1'(b_{21}^1{\mathbf w}_1+b_{22}^1{\mathbf w}_2+b_{23}^1{\mathbf w}_3)+\beta\bar{\mathbf w}_2'(b_{21}^2{\mathbf w}_1+b_{22}^2{\mathbf w}_2+b_{23}^2{\mathbf w}_3)+\\
\vspace{-0.3cm}\\
\hspace{7cm}\beta\bar{\mathbf w}_3'(b_{21}^3{\mathbf w}_1+b_{22}^3{\mathbf w}_2+b_{23}^3{\mathbf w}_3))\equiv 0,\\
\vspace{-0.1cm}\\
\Im(\gamma\bar{\mathbf w}_1'(b_{11}^1{\mathbf w}_1+b_{12}^1{\mathbf w}_2+b_{13}^1{\mathbf w}_3)+\delta\bar{\mathbf w}_2'(b_{11}^2{\mathbf w}_1+b_{12}^2{\mathbf w}_2+b_{13}^2{\mathbf w}_3)+\\
\vspace{-0.3cm}\\
\hspace{7cm}\delta\bar{\mathbf w}_3'(b_{11}^3{\mathbf w}_1+b_{12}^3{\mathbf w}_2+b_{13}^3{\mathbf w}_3))\equiv 0,
\end{array}
$$
which yields
\begin{equation}
\begin{array}{llllll}
b_{21}^1=0, & b_{22}^1=0, & b_{23}^1=0, & \beta b_{21}^2=0, & \beta b_{22}^2=0, & \beta b_{23}^2=0,\\
\vspace{-0.3cm}\\
\beta b_{21}^3=0, & \beta b_{22}^3=0, & \beta b_{23}^3=0, & \gamma b_{11}^1=0, & \gamma b_{12}^1=0, & \gamma b_{13}^1=0,\\
\vspace{-0.3cm}\\
\delta b_{11}^2=0, & \delta b_{12}^2=0, & \delta b_{13}^2=0, & \delta b_{11}^3=0, & \delta b_{12}^3=0, & \delta b_{13}^3=0.
\end{array}\label{zero11}
\end{equation}
If each of $\beta$, $\gamma$, $\delta$ is non-zero, it follows from (\ref{zero11}) that $b=0$, therefore $B_{{\mathbf x}}=0$ for all ${\mathbf x}\in\RR^2$, and the requirement that $B_{{\mathbf x}}$ is associated to $A_{{\mathbf x}}$ with respect to ${\mathcal H}$ for every ${\mathbf x}\in\RR^2$ (see condition (i) in Theorem \ref{descrg1}) implies $a=0$. Thus, ${\mathfrak g}_1=0$ as required.

Assume now that $\beta=0$, hence each of $\gamma$, $\delta$ is non-zero, and by (\ref{zero11}) we have
$$
\begin{array}{l}
b_{11}^1=0,\,\,b_{12}^1=0,\,\,b_{13}^1=0,\,\,
b_{21}^1=0, \,\,b_{22}^1=0,\,\,b_{23}^1=0,\\
\vspace{-0.3cm}\\
b_{11}^2=0, \,\, b_{12}^2=0,\,\,b_{13}^2=0,\,\,b_{11}^3=0, \,\, b_{12}^3=0,\,\,b_{13}^3=0.
\end{array}
$$
Thus,
\begin{equation}
\begin{array}{l}
b(z,w)=\left(0,b_{21}^2z_2w_1+b_{22}^2z_2w_2+b_{23}^2z_2w_3,\right.\\
\vspace{-0.4cm}\\
\hspace{6cm}\left.b_{21}^3z_2w_1+b_{22}^3z_2w_2+b_{23}^3z_2w_3\right)
\end{array}\label{formb14}
\end{equation}
and
\begin{equation}
\begin{array}{l}
B_{{\mathbf x}}(w)=\displaystyle\frac{1}{2}\left(0,b_{21}^2{\mathbf x}_2w_1+b_{22}^2{\mathbf x}_2w_2+b_{23}^2{\mathbf x}_2w_3,\right.\\
\vspace{-0.4cm}\\
\hspace{6cm}\left.b_{21}^3{\mathbf x}_2w_1+b_{22}^3{\mathbf x}_2w_2+b_{23}^3{\mathbf x}_2w_3\right).
\end{array}\label{formb1}
\end{equation}

We will now utilize the requirement that $B_{{\mathbf x}}$ is associated to $A_{{\mathbf x}}$ with respect to ${\mathcal H}$ for every ${\mathbf x}\in\RR^2$ as in (i) in Theorem \ref{descrg1}. On the one hand, from (\ref{forma}) we have
\begin{equation}
A_{{\mathbf x}}{\mathcal H}(w,w')=(a_{11}^1{\mathbf x}_1\alpha\bar w_1w_1',a_{22}^2{\mathbf x}_2(\gamma \bar w_1w_1'+\delta \bar w_2w_2'+\delta \bar w_3w_3')).\label{left11}
\end{equation}
On the other hand, from (\ref{formb1}) one obtains
\begin{equation}
\begin{array}{l}
\displaystyle{\mathcal H}(B_{{\mathbf x}}(w),w')+{\mathcal H}(w,B_{{\mathbf x}}(w'))=\frac{1}{2}\left(0,\delta(\bar b_{21}^2{\mathbf x}_2\bar w_1+\bar b_{22}^2{\mathbf x}_2\bar w_2+\right.\\
\vspace{-0.3cm}\\
\hspace{0.5cm}\left.\bar b_{23}^2{\mathbf x}_2\bar w_3)w_2'+\delta(\bar b_{21}^3{\mathbf x}_2\bar w_1+\bar b_{22}^3{\mathbf x}_2\bar w_2+\bar b_{23}^3{\mathbf x}_2\bar w_3)w_3'+\delta\bar w_2(b_{21}^2{\mathbf x}_2w_1'+\right.\\
\vspace{-0.1cm}\\
\hspace{1.5cm}\left.b_{22}^2{\mathbf x}_2w_2'+b_{23}^2{\mathbf x}_2w_3')+\delta\bar w_3(b_{21}^3{\mathbf x}_2w_1'+b_{22}^3{\mathbf x}_2w_2'+b_{23}^3{\mathbf x}_2w_3')\right).
 \end{array}\label{right11}
\end{equation}
Comparing (\ref{left11}) and  (\ref{right11}), we deduce
\begin{equation}
a_{11}^1=0,\,\,a_{22}^2=0,\,\,b_{21}^2=0,\,\,b_{21}^3=0,\label{zero41}
\end{equation}
in particular, $a=0$.

By (\ref{zero41}), formula (\ref{formb14}) simplifies as
$$
b(z,w)=(0,b_{22}^2z_2w_2+b_{23}^2z_2w_3,b_{22}^3z_2w_2+b_{23}^3z_2w_3),
$$
and we will now use condition (iii) in Theorem \ref{descrg1}. We have
$$
\begin{array}{l}
b({\mathcal H}(w',w''),w'')=\left(0,(\gamma\bar w_1'w_1''+\delta\bar w_2'w_2''+\delta\bar w_3'w_3'')(b_{22}^2w_2''+b_{23}^2w_3''),\right.\\
\vspace{-0.3cm}\\
\hspace{4cm}\left.(\gamma\bar w_1'w_1''+\delta\bar w_2'w_2''+\delta\bar w_3'w_3'')(b_{22}^3w_2''+b_{23}^3w_3'')\right),
\end{array}
$$
therefore
\begin{equation}
\begin{array}{l}
{\mathcal H}(w,b({\mathcal H}(w',w''),w''))=\left(0,\delta\bar w_2(\gamma\bar w_1'w_1''+\delta\bar w_2'w_2''+\delta\bar w_3'w_3'')\times\right.\\
\vspace{-0.3cm}\\
\hspace{0.3cm}\left.(b_{22}^2w_2''+b_{23}^2w_3'')+\delta\bar w_3(\gamma\bar w_1'w_1''+\delta\bar w_2'w_2''+\delta\bar w_3'w_3'')(b_{22}^3w_2''+b_{23}^3w_3'')\right).
\end{array}\label{complleft}
\end{equation}
On the other hand
$$
\begin{array}{l}
b({\mathcal H}(w'',w),w')=\left(0,(\gamma\bar w_1''w_1+\delta\bar w_2''w_2+\delta\bar w_3''w_3)(b_{22}^2w_2'+b_{23}^2w_3'),\right.\\
\vspace{-0.3cm}\\
\hspace{4cm}\left.(\gamma\bar w_1''w_1+\delta\bar w_2''w_2+\delta\bar w_3''w_3)(b_{22}^3w_2'+b_{23}^3w_3')\right),
\end{array}
$$
hence
\begin{equation}
\begin{array}{l}
{\mathcal H}(b({\mathcal H}(w'',w),w'),w'')=\left(0,\delta(\gamma w_1''\bar w_1+\delta w_2''\bar w_2+\delta w_3''\bar w_3)\times\right.\\
\vspace{-0.3cm}\\
\hspace{2cm}\left.(\bar b_{22}^2\bar w_2'+\bar b_{23}^2\bar w_3')w''_2+\delta(\gamma w_1''\bar w_1+\delta w_2''\bar w_2+\delta w_3''\bar w_3)\times\right.\\
\vspace{-0.3cm}\\
\hspace{7cm}\left.(\bar b_{22}^3\bar w_2'+\bar b_{23}^3\bar w_3')w''_3\right).\label{complright}
\end{array}
\end{equation}
Comparing (\ref{complleft}) and (\ref{complright}) we see that $b_{22}^2=0$, $b_{23}^2=0$, $b_{22}^3=0$, $b_{23}^3=0$ and therefore $b=0$. This proves that ${\mathfrak g}_1=0$ as claimed.

Suppose now that $\gamma=0$, hence each of $\beta,\delta$ is non-zero. By (\ref{zero11}) we then have
$$
\begin{array}{l}
b_{21}^1=0,\,\,b_{22}^1=0,\,\,b_{23}^1=0,\,\,
b_{11}^2=0, \,\,b_{12}^2=0,\,\,b_{13}^2=0,\,\,b_{21}^2=0,\,\,b_{22}^2=0,\\
\vspace{-0.3cm}\\
b_{23}^2=0,\,\,b_{11}^3=0,\,\,b_{12}^3=0, \,\, b_{13}^3=0,\,\,b_{21}^3=0,\,\,b_{22}^3=0,\,\,b_{23}^3=0,
\end{array}
$$
and (\ref{imrel11}) implies
\begin{equation}
\Im b_{11}^1=0.\label{imrel12}
\end{equation}
Thus,
\begin{equation}
B_{{\mathbf x}}(w)=\displaystyle\frac{1}{2}\left(b_{11}^1{\mathbf x}_1w_1+b_{12}^1{\mathbf x}_1w_2+b_{13}^1{\mathbf x}_1w_3, 0, 0\right).\label{formb134}
\end{equation}

We will now recall that $B_{{\mathbf x}}$ is associated to $A_{{\mathbf x}}$ with respect to ${\mathcal H}$ for every ${\mathbf x}\in\RR^2$. On the one hand, from (\ref{forma}) we have
\begin{equation}
A_{{\mathbf x}}{\mathcal H}(w,w')=(a_{11}^1{\mathbf x}_1(\alpha\bar w_1w_1'+\beta \bar w_2w_2'+\beta \bar w_3w_3'),a_{22}^2{\mathbf x}_2(\delta \bar w_2w_2'+\delta \bar w_3w_3')).\label{left134}
\end{equation}
On the other hand, from (\ref{formb134}) one obtains
\begin{equation}
\begin{array}{l}
\displaystyle{\mathcal H}(B_{{\mathbf x}}(w),w')+{\mathcal H}(w,B_{{\mathbf x}}(w'))=\frac{1}{2}\left(\alpha(\bar b_{11}^1{\mathbf x}_1\bar w_1+\bar b_{12}^1{\mathbf x}_1\bar w_2+\right.\\
\vspace{-0.3cm}\\
\hspace{2cm}\left.\bar b_{13}^1{\mathbf x}_1\bar w_3)w_1'+\alpha\bar w_1(b_{11}^1{\mathbf x}_1w_1'+b_{12}^1{\mathbf x}_1w_2'+b_{13}^1{\mathbf x}_1w_3'),0\right).
\end{array}\label{right134}
\end{equation}
Comparing (\ref{left134}) and  (\ref{right134}), we see
\begin{equation}
a_{11}^1=0,\,\,a_{22}^2=0,\,\,\Re b_{11}^1=0,\,\,b_{12}^1=0,\,\,b_{13}^1=0.\label{zero414}
\end{equation}
By (\ref{imrel12}), (\ref{zero414}) it follows that $a=0$, $b=0$, which shows that ${\mathfrak g}_1=0$.

To complete the proof, notice that the case $\delta=0$ is obtained from the case $\beta=0$ by permutation of variables.\end{proof}

Finally, we consider the domain $D_6$.

\begin{proposition}\label{g1d6} For ${\mathfrak g}={\mathfrak g}(D_6)$ one has $\dim{\mathfrak g}_1=1$.
\end{proposition}

\begin{proof} As in the proof of Lemma \ref{g12d6}, we assume that $v=(1,1,0)$. We will utilize Theorem \ref{descrg1} for the cone $\Omega_3$ and the $\Omega_3$-Hermitian form ${\mathcal H}$ defined in (\ref{formh444}).

Let $b:\CC^3\times\CC\to\CC$ be a $\CC$-bilinear map:
$$
b(z,w)=(b_1z_1+b_2z_2+b_3z_3)w,
$$
where $b_j\in\CC$. For every fixed pair ${\mathbf w},{\mathbf w}'\in\CC$ we compute
$$
{\mathcal H}({\mathbf w}',b(x,{\mathbf w}))=\left(\bar {\mathbf w}'{\mathbf w}(b_1x_1+b_2x_2+b_3x_3),\bar {\mathbf w}'{\mathbf w}(b_1x_1+b_2x_2+b_3x_3),0\right),
$$
with $x\in\RR^3$. Then from (ii) of Theorem \ref{descrg1} we obtain
$$
\begin{array}{l}
B_{{\mathbf w},{\mathbf w}'}(x)=\left(\Im(b_1\bar {\mathbf w}'{\mathbf w})x_1+\Im(b_2\bar {\mathbf w}'{\mathbf w})x_2+\Im(b_3\bar {\mathbf w}'{\mathbf w})x_3,\right.\\
\vspace{-0.3cm}\\
\hspace{4cm}\left.\Im(b_1\bar {\mathbf w}'{\mathbf w})x_1+\Im(b_2\bar {\mathbf w}'{\mathbf w})x_2+\Im(b_3\bar {\mathbf w}'{\mathbf w})x_3,0\right).
\end{array}
$$
By (\ref{algomega3}), the condition that this map lies in ${\mathfrak g}(\Omega_3)$ for all ${\mathbf w},{\mathbf w}'\in\CC$ immediately yields 
$$
\Im(b_1\bar {\mathbf w}'{\mathbf w})\equiv 0,\quad \Im(b_2\bar {\mathbf w}'{\mathbf w})\equiv 0,\quad \Im(b_3\bar {\mathbf w}'{\mathbf w})\equiv 0,
$$
hence $b=0$.

Next, consider a symmetric $\RR$-bilinear form on $\RR^3$ with values in $\RR^3$:
$$
\begin{array}{l}
a(x,x)=\left(a_{11}^1x_1^2+2a_{12}^1x_1x_2+2a_{13}^1x_1x_3+a_{22}^1x_2^2+2a_{23}^1x_2x_3+a_{33}^1x_3^2,\right.\\
\vspace{-0.3cm}\\
\hspace{1.6cm}\left.a_{11}^2x_1^2+2a_{12}^2x_1x_2+2a_{13}^2x_1x_3+a_{22}^2x_2^2+2a_{23}^2x_2x_3+a_{33}^2x_3^2,\right.\\
\vspace{-0.3cm}\\
\hspace{1.6cm}\left.a_{11}^3x_1^2+2a_{12}^3x_1x_2+2a_{13}^3x_1x_3+a_{22}^3x_2^2+2a_{23}^3x_2x_3+a_{33}^3x_3^2
\right),
\end{array}
$$
where $a_{ij}^k\in\RR$. Then for a fixed ${\mathbf x}\in\RR^3$ from (\ref{idents1}) we compute
$$
\begin{array}{l} 
A_{{\mathbf x}}(x)=\left(a_{11}^1{\mathbf x}_1x_1+a_{12}^1{\mathbf x}_1x_2+a_{12}^1{\mathbf x}_2x_1+a_{13}^1{\mathbf x}_1x_3+a_{13}^1{\mathbf x}_3x_1+a_{22}^1{\mathbf x}_2x_2+\right.\\
\vspace{-0.3cm}\\
\hspace{0.7cm}\left.a_{23}^1{\mathbf x}_2x_3+a_{23}^1{\mathbf x}_3x_2+a_{33}^1{\mathbf x}_3x_3,
a_{11}^2{\mathbf x}_1x_1+a_{12}^2{\mathbf x}_1x_2+a_{12}^2{\mathbf x}_2x_1+a_{13}^2{\mathbf x}_1x_3+\right.\\
\vspace{-0.3cm}\\
\hspace{0.7cm}\left.a_{13}^2{\mathbf x}_3x_1+a_{22}^2{\mathbf x}_2x_2+a_{23}^2{\mathbf x}_2x_3+a_{23}^2{\mathbf x}_3x_2+a_{33}^2{\mathbf x}_3x_3,
a_{11}^3{\mathbf x}_1x_1+a_{12}^3{\mathbf x}_1x_2+\right.\\
\vspace{-0.3cm}\\
\hspace{0.7cm}\left.a_{12}^3{\mathbf x}_2x_1+a_{13}^3{\mathbf x}_1x_3+a_{13}^3{\mathbf x}_3x_1+a_{22}^3{\mathbf x}_2x_2+a_{23}^3{\mathbf x}_2x_3+a_{23}^3{\mathbf x}_3x_2+a_{33}^3{\mathbf x}_3x_3\right)=\\
\vspace{-0.3cm}\\
\hspace{0.7cm}\left((a_{11}^1{\mathbf x}_1+a_{12}^1{\mathbf x}_2+a_{13}^1{\mathbf x}_3)x_1+(a_{12}^1{\mathbf x}_1+a_{22}^1{\mathbf x}_2+a_{23}^1{\mathbf x}_3)x_2+(a_{13}^1{\mathbf x}_1+a_{23}^1{\mathbf x}_2+\right.\\
\vspace{-0.3cm}\\
\hspace{0.7cm}\left.a_{33}^1{\mathbf x}_3)x_3,
(a_{11}^2{\mathbf x}_1+a_{12}^2{\mathbf x}_2+a_{13}^2{\mathbf x}_3)x_1+(a_{12}^2{\mathbf x}_1+a_{22}^2{\mathbf x}_2+a_{23}^2{\mathbf x}_3)x_2+(a_{13}^2{\mathbf x}_1+\right.\\
\vspace{-0.3cm}\\
\hspace{0.7cm}\left.a_{23}^2{\mathbf x}_2+a_{33}^2{\mathbf x}_3)x_3,
(a_{11}^3{\mathbf x}_1+a_{12}^3{\mathbf x}_2+a_{13}^3{\mathbf x}_3)x_1+(a_{12}^3{\mathbf x}_1+a_{22}^3{\mathbf x}_2+a_{23}^3{\mathbf x}_3)x_2+\right.\\
\vspace{-0.3cm}\\
\hspace{0.7cm}\left(a_{13}^3{\mathbf x}_1+a_{23}^3{\mathbf x}_2+a_{33}^3{\mathbf x}_3)x_3\right),
\end{array}
$$
where $x\in\RR^3$. By (\ref{algomega3}), the condition that this map lies in ${\mathfrak g}(\Omega_3)$ for every ${\mathbf x}\in\RR^3$ is equivalent to
\begin{equation}
\begin{array}{l}
a_{11}^1{\mathbf x}_1+a_{12}^1{\mathbf x}_2+a_{13}^1{\mathbf x}_3\equiv a_{12}^2{\mathbf x}_1+a_{22}^2{\mathbf x}_2+a_{23}^2{\mathbf x}_3\equiv\\
\vspace{-0.3cm}\\
\hspace{3.9cm} a_{13}^3{\mathbf x}_1+a_{23}^3{\mathbf x}_2+a_{33}^3{\mathbf x}_3,\\
\vspace{-0.3cm}\\
 a_{12}^1{\mathbf x}_1+a_{22}^1{\mathbf x}_2+a_{23}^1{\mathbf x}_3\equiv a_{11}^2{\mathbf x}_1+a_{12}^2{\mathbf x}_2+a_{13}^2{\mathbf x}_3,\\
 \vspace{-0.3cm}\\
 a_{13}^1{\mathbf x}_1+a_{23}^1{\mathbf x}_2+a_{33}^1{\mathbf x}_3\equiv a_{11}^3{\mathbf x}_1+a_{12}^3{\mathbf x}_2+a_{13}^3{\mathbf x}_3,\\
  \vspace{-0.3cm}\\
 a_{13}^2{\mathbf x}_1+a_{23}^2{\mathbf x}_2+a_{33}^2{\mathbf x}_3\equiv -(a_{12}^3{\mathbf x}_1+a_{22}^3{\mathbf x}_2+a_{23}^3{\mathbf x}_3).
\end{array}\label{rel888}
\end{equation}

Further, recalling that any map $b:\CC^3\times\CC\to\CC$ as above is zero, we will utilize the condition that the zero matrix is associated to $A_{{\mathbf x}}$ for every ${\mathbf x}\in\RR^3$ as in (i) of Theorem \ref{descrg1}. This condition means
\begin{equation}
\begin{array}{l}
a_{12}^1{\mathbf x}_1+a_{22}^1{\mathbf x}_2+a_{23}^1{\mathbf x}_3\equiv -(a_{11}^1{\mathbf x}_1+a_{12}^1{\mathbf x}_2+a_{13}^1{\mathbf x}_3),\\
\vspace{-0.3cm}\\
a_{13}^2{\mathbf x}_1+a_{23}^2{\mathbf x}_2+a_{33}^2{\mathbf x}_3\equiv a_{11}^3{\mathbf x}_1+a_{12}^3{\mathbf x}_2+a_{13}^3{\mathbf x}_3.
\end{array}\label{rel8888}
\end{equation}

Combining identities (\ref{rel888}) and (\ref{rel8888}), we obtain the following relations for the coefficients of the form $a$:
$$
\begin{array}{l}
a_{11}^1=a_{12}^2=a_{13}^3,\, a_{12}^1=a_{22}^2=a_{23}^3,\, a_{13}^1=a_{23}^2=a_{33}^3,\, a_{12}^1=a_{11}^2,\, a_{22}^1=a_{12}^2,\, a_{23}^1=a_{13}^2,\\
\vspace{-0.3cm}\\
a_{13}^1=a_{11}^3,\, a_{23}^1=a_{12}^3,\, a_{33}^1=a_{13}^3,\, a_{13}^2=-a_{12}^3,\, a_{23}^2=-a_{22}^3,\, a_{33}^2=-a_{23}^3,\, a_{12}^1=-a_{11}^1,\\
\vspace{-0.3cm}\\ 
a_{22}^1=-a_{12}^1,\, a_{23}^1=-a_{13}^1,\, a_{13}^2=a_{11}^3,\, a_{23}^2=a_{12}^3,\, a_{33}^2=a_{13}^3.
\end{array}
$$
By the above relations, each coefficient of $a$ either is zero or is equal to $\pm a_{11}^1$ as follows:
$$
\begin{array}{l}
a_{12}^1=-a_{11}^1,\, a_{13}^1=0,\, a_{22}^1=a_{11}^1,\, a_{23}^1=0,\, a_{33}^1=a_{11}^1,\\
\vspace{-0.3cm}\\
\hspace{1cm} a_{11}^2=-a_{11}^1,\, a_{12}^2=a_{11}^1,\, a_{13}^2=0,\, a_{22}^2=-a_{11}^1,\, a_{23}^2=0, \\
\vspace{-0.3cm}\\ 
 \hspace{2cm}a_{33}^2=a_{11}^1,\, a_{11}^3=0,\, a_{12}^3=0,\, a_{13}^3=a_{11}^1,\, a_{22}^3=0,\, a_{23}^3=-a_{11}^1,\, a_{33}^3=0.
\end{array}
$$
Therefore
$$
a(x,x)=a_{11}^1\left((x_1-x_2)^2+x_3^2,-(x_1-x_2)^2+x_3^2,2(x_1-x_2)x_3\right).
$$
This shows that $\dim{\mathfrak g}_1=1$ as required.\end{proof}

\end{document}